\documentclass[11pt]{amsart}
\usepackage[a4paper,margin=2.5cm]{geometry}
\usepackage{amsaddr,hyperref,mathtools,amsmath,amsthm,amssymb,amscd,amsfonts,color,tikz-cd,pgfplots,diagbox,dsfont}
\usetikzlibrary{arrows}

\theoremstyle{plain}
\newtheorem{theorem}{Theorem}[section]
\newtheorem{proposition}[theorem]{Proposition}
\newtheorem{lemma}[theorem]{Lemma}
\newtheorem{corollary}[theorem]{Corollary}
\numberwithin{equation}{section}

\theoremstyle{definition}
\newtheorem{definition}[theorem]{Definition}
\newtheorem{remark}[theorem]{Remark}
\newtheorem{example}[theorem]{Example}

\newtheorem{conjecture}[theorem]{Conjecture}
\newtheorem{problem}[theorem]{Problem}

\newcommand{\C}{\mathbb{C}}

\newcommand{\bS}{\mathbb{S}}
\newcommand{\cP}{\mathcal{P}}
\newcommand{\cQ}{\mathcal{Q}}

\newcommand{\cC}{\mathcal{C}}

\newcommand{\cI}{\mathcal{I}}

\newcommand{\cF}{\mathcal{F}}
\newcommand{\cR}{\mathcal{R}}

\newcommand{\cG}{\mathcal{G}}

\newcommand{\fC}{\mathsf{C}}
\newcommand{\B}{\mathcal{B}}
\newcommand{\V}{\mathcal{V}}
\newcommand{\alt}{\mathsf{a}}
\newcommand{\bE}{\mathbf{E}}

\newcommand{\Q}{\mathbb{Q}}

\newcommand{\R}{\mathbb{R}}
\newcommand{\Z}{\mathbb{Z}}

\newcommand{\ie}{\textit{i}.\textit{e}. }

\newcommand{\Hoch}{\mathds{H}\mathrm{och}}
\newcommand{\Alt}{\mathrm{Alt}}

\DeclareMathOperator{\Lk}{Lk}

\DeclareMathOperator{\odd}{odd}
\DeclareMathOperator{\even}{even}
\DeclareMathOperator{\conv}{conv}
\DeclareMathOperator{\row}{row}

\begin{document}
\title[Real toric manifolds associated with chordal nestohedra]{Real toric manifolds associated with chordal nestohedra}

\author[S. Choi and Y. Yoon]{Suyoung Choi$^1$, Younghan Yoon$^1$}

\address{$^1$Department of Mathematics, Ajou University, 206, World cup-ro, Yeongtong-gu, Suwon 16499, Republic of Korea}
\email{schoi@ajou.ac.kr, younghan300@ajou.ac.kr}

\date{\today}
\subjclass[2020]{57S12, 14M25, 57N65, 05A05, 52B22}

\keywords{homology group, toric topology, real toric varieties, chordal nestohedra, poset topology, EL-shellability, $\B$-permutations, alternating permutations, chordal graphs, graph associahedra, Stanley-Pitman polytopes, Hochschild polytopes}

\thanks{This work was supported by the National Research Foundation of Korea Grant funded by
the Korean Government (RS-2025-00521982).}

\begin{abstract}
    This paper investigates the rational Betti numbers of real toric manifolds associated with chordal nestohedra.
    We consider the poset topology of a specific poset induced from a chordal building set, and show its EL-shellability.
    Based on this, we present an explicit description using alternating $\B$-permutations for a chordal building set $\B$, transforming the computing Betti numbers into a counting problem.
    This approach allows us to compute the $a$-number of a finite simple graph through permutation counting when the graph is chordal.
    In addition, we provide detailed computations for specific cases such as real Hochschild varieties corresponding to Hochschild polytopes.
\end{abstract}
\maketitle
\tableofcontents

\section{Introduction}

A \emph{toric variety}~$X$ is a normal complex algebraic variety containing the algebraic torus~$(\C\setminus\{O\})^n$ on itself extends to an action on $X$.
A smooth compact toric variety~$X$ is called a \emph{toric manifold}, and the real locus~$X^\R$ of $X$, that is, the fixed point set of $X$ by the canonical involution induced from a complex conjugation, is called a \emph{real toric manifold}.
By the fundamental theorem of toric geometry, there is a one-to-one correspondence between the class of $n$-dimensional toric manifolds and the class of non-singular complete fans in $\R^n$.
It is well known \cite{Cox_toric_book} that a toric variety is projective if and only if its associated fan is the normal fan of some simple polytope.
An $n$-dimensional simple polytope is said to be \emph{Delzant} if for each vertex $v$ the outward normal vectors of the facets containing $v$ can be chosen so that they form an integral basis for $\Z^n$.
Since the normal fan of a Delzant polytope~$P$ is a non-singular complete fan, $P$ induces the projective toric manifold and the corresponding real toric manifold.

Our investigation centers on some family of Delzant polytopes known as \emph{nestohedra}, which has been well studied in various literature, such as \cite{Postnikov2009}, \cite{Postnikov2008}, \cite{Choi-Park2015}, \cite{Pilaud2017},  \cite{Vladimir2017}, \cite{Ivan2017}, \cite{Vladimir-Tanja2017}, \cite{Dotsenko-Shadrin-Vallette2019}, and~\cite{Choi-Hwang2023}.
For a finite set~$S$ of natural numbers, a \emph{building set} on $S$ is a collection of nonempty subsets of~$S$ which includes every singleton of~$S$ and satisfies the following condition:
$$
  I,J \in \B \text{ and } I \cap J \neq \emptyset \Rightarrow I \cup J \in \B.
$$
The notion of building sets is defined in~\cite{DeConcini-Procesi1995} with respect to a hyperplane arrangement of vector spaces.
An inclusion-maximal subset in $\B$ is called a \emph{connected component} of $\B$.
If there is a unique connected component of $\B$, a building set $\B$ is said to be \emph{connected}.

A \emph{nestohedron} $P_{\B}$ associated with $\B$ is defined as the Minkowski sum of simplices
$$
    P_\B = \sum_{I \in \B}{\conv(I)},
$$
where $\conv (I)$ is the convex hull of the $i$th coordinate vectors for all $i \in I$.

Since $P_\B$ is Delzant as shown in \cite{Postnikov2009}, it induces a projective toric manifold.
We refer to this toric manifold as the \emph{toric manifold associated with a building set~$\B$} and denoted it by~$X_{\B}$, and its corresponding real toric manifold by~$X^\R_{\B}$.
For a connected building set~$\B$ on $S$ of cardinality~$n+1$, the associated toric manifold~$X_\B$ is of dimension~$n$.
By \cite[Remark~6.7]{Postnikov2008}, if $\B_1,\B_2,\ldots,\B_\ell$ are the connected components of a building set $\B$, then $P_\B$ is isomorphic to the direct product
$$
P_{\B_1} \times \cdots \times P_{\B_\ell}.
$$
Moreover, by \cite[Theorem~2.4.7]{Cox_toric_book}, the associated toric manifold $X_{\B}$ is isomorphic to the product~$X_{\B_1} \times \cdots \times X_{\B_\ell}$ as projective toric manifolds.
Hence, to study toric manifolds or real toric manifolds associated with building sets, it is sufficient to consider the case when the building set is connected.

As established in \cite{Davis-Januszkiewicz1991}, for a building set $\B$, the rational Betti numbers of $X_\B$ and the $\Z_2$-Betti numbers of $X^\R_\B$ are completely determined from the numbers of faces of $P_\B$, where $\Z_2$ is the field of order~$2$.
Particularly, in the case where $\B$ is a \emph{chordal building set} on $S$, that is,
$$
I = \{i_1 < \cdots <i_r\} \in \B \Rightarrow \{i_s,\ldots,i_r\} \in \B \text{ for each $1 < s < r$},
$$
there is an explicit description, established in  \cite{Postnikov2008}, of the face numbers of $P_\B$ in terms of permutations on $S$.
Note that this description confirms that the Gal's conjecture \cite{Gal2005} is true for chordal nestohedra.
To discuss more details, we briefly introduce some notions for a building set $\B$ on $S$: 
\begin{itemize}
  \item For any subset $I$ of $S$, $\B \vert_{I} = \{J \in \B \colon J \subset I\}$ is a \emph{restricted building set} of $\B$ to $I$.
  \item A (one-line notation) permutation $x = (x_1x_2\cdots x_{k})$ on $S$ is a \emph{$\B$-permutation} if, for each $1 \leq i \leq k$, there is $J \in \B \vert_{\{x_1,x_2,\ldots,x_{i}\}}$ such that $J$ contains both
      $x_i$ and $\max{\{x_1,x_2,\ldots,x_i\}}$.
\end{itemize}

From now on in this paper, for a topological space~$M$, $\beta_k(M)$ will simply be referred to as the $k$th rational Betti number of $M$.
From \cite{Davis-Januszkiewicz1991} and \cite{Postnikov2008}, for a connected chordal building set~$\B$ on $[n+1] = \{1,2,\ldots,n+1\}$, the $k$th Betti numbers of $X_\B$ is
\begin{equation}\label{B-perm}
  \beta_k(X_\B) = \begin{cases}
                     \#\text{$\B$-permutations with $\frac{k}{2}$ descents}, & \mbox{if $0 \leq k \leq 2n$ is even,} \\
                     0, & \mbox{otherwise},
                   \end{cases} 
\end{equation}
where a \emph{descent} of a permutation $(x_1x_2\cdots x_n)$ on $[n+1] =\{1,\ldots, n+1\}$ is $1 \leq i < n$ where $x_i > x_{i+1}$.

It is also natural to consider Betti number formulas for real toric manifold~$X^\R_\B$ associated with~$\B$.
Although the $\Z_2$-Betti numbers of $X_\B^\R$ can also be computed in the same manner as in~\eqref{B-perm}, 
the computation of its rational Betti numbers is significantly more complicated in general.

One typical and important example of a chordal building set is the maximal building set~$\B$ on~$[n+1]$, that is, $\B$ is the set of nonempty subsets of~$[n+1]$.
The corresponding polytope is known as a \emph{permutohedron} and has dimension $n$.
Likewise, the associated toric variety, having a complex dimension of $n$, is referred to as the \emph{Coxeter variety} of type $A_n$, denoted by $X_{A_n}$.
In many literature, it is also known as a \emph{permutohedral variety}.
Furthermore, its real counterpart~$X_\B^\R$ of dimension $n$ is called a \emph{real Coxeter variety} $X_{A_n}^\R$ of type~$A_n$ or a \emph{real permutohedral variety}.
In his paper \cite{Henderson2012}, Henderson employed a geometric approach to compute the $k$th Betti numbers of $X_{A_n}^\R$, which can be expressed as follows:
$$
  \beta_k(X^{\R}_{A_n}) = \binom{n+1}{2k} a_{2k},
$$
where $a_{2k}$ denotes the $2k$th Euler zigzag number (A000111 in \cite{oeis}), signifying the number of alternating permutations of length $2k$.
Later, this formula is revisited by several studies including \cite{Suciu2012}, \cite{Choi-Park2015}, \cite{Cho-Choi-Kaji2019}, which explored poset topologies based on the formula established in \cite{ST2012} and \cite{Choi-Park2017_torsion}.
Additionally, the multiplicative structure of the cohomology ring of $X_{A_n}^\R$ has been completely described in terms of alternating permutations in our previous work \cite{Choi-Yoon2023}.

One observes that a $\B$-permutation for the
 maximal building set $\B$ on $[n+1]$ is an ordinary permutation.
In other words, in this case, the Betti number of $X_\B^\R = X_{A_n}^\R$ can be described using alternating ($\B$-)permutations, whereas that of $X_\B = X_{A_n}$ is describable using permutations as in \eqref{B-perm}.
This observation motivates us to investigate whether  the Betti numbers of $X_\B^\R$ for a general chordal building set $\B$ can similarly be described by alternating $\B$-permutations, where an \emph{alternating $\B$-permutation} is a $\B$-permutation $x=(x_1x_2\cdots x_n)$ such that $x_1 > x_2 < x_3 > \cdots$.

Now, we introduce the main result of this paper, and the proof will be given in Section~\ref{sec:proof_of_main_theorem}.

\begin{theorem}\label{main1}
    Let $\B$ be a connected chordal building set on $[n+1] = \{1,2,\ldots,n+1\}$.
    The $k$th Betti number of $X^\R_{\B}$ is
    $$
    \beta_{k}(X^\R_{\B}) = \sum_{I \in {[n+1] \choose 2k}}{\# \text{alternating $\B\vert_I$-permutations}}.
    $$
\end{theorem}

It should be noted that the condition of the chordality of a building set is necessary.
An example introduced in Section~\ref{subsection:Cyclohedra} demonstrates that Theorem~\ref{main1} does not hold for an arbitrary building set even for small dimensions~$n$.

To prove the main theorem, for a chordal building set $\B$ on $[n+1]$, we need to consider the specific poset~$\hat{\cP}_\B$ consisting of all subsets $I$ of $[n+1]$ such that $\B\vert_I$ does not admit odd components.
The key part is showing that the poset~$\hat{\cP}_\B$ is EL-shellable in Section~\ref{section3}.
EL-shellability has been extensively studied in the literature, such as in \cite{Muhle2015}, \cite{Can2019}, \cite{Li2021}, and \cite{Li2024}.
The authors wish to emphasize that this provides a non-trivial example of naturally occurring EL-shellable complexes.
Due to EL-shellability of $\hat{\cP}_\B$, in Section~\ref{sec:alt_B_perm}, we deal with the poset topology of $\hat{\cP}_\B$ in terms of alternating $\B$-permutations.
The full proof of Theorem~\ref{main1} will be given in Section~\ref{sec:proof_of_main_theorem}.

In Section~\ref{section4}, we apply the main result in several remarkable chordal nestohedra, such as \emph{permutohedra}, \emph{associahedra} (or \emph{Stasheff polytopes}), \emph{stellohedra}, \emph{Stanley-Pitman polytopes} \cite{Stanley-Pitman}, and \emph{Hochschild polytopes} \cite{Vincent-Daria2023}.
See Figure~\ref{category} for an illustrated diagram.
\begin{figure}
\centering
\begin{tikzpicture}
    \draw[rounded corners=2mm] (-6,-3) rectangle (6,3) node [below left] {Nestohedrda};
    \draw[rounded corners=2mm] (1.8,-2) rectangle  (-5.5,2.7) node [below right] {Chordal nestohedrda};
    \draw[rounded corners=2mm] (-1.5,-2.5) rectangle (5.5,1) node [below left] {Graph associahedra};
    \draw[rounded corners=2mm] (1.5,-0.2) rectangle (-5,2) node [below right] {Hochschild Polytopes};
    \draw[rounded corners=2mm] (-4.7,0) rectangle (-1.7,1.3);
    \node (1) at (-3.2,0.8) {Stanley-Pitman};
    \node (1) at (-3.2,0.4) {Polytopes};
    \draw[rounded corners=2mm] (-1.3,0.1) rectangle (1.4,0.7) node [below left] {Permutohedra};
    \draw[rounded corners=2mm] (-1.2,-1) rectangle (1.3,-0.4) node [below left] {Associahedra};
    \draw[rounded corners=2mm] (-1.1,-1.8) rectangle (1.1,-1.2) node [below left] {Stellohedra};
    \draw[rounded corners=2mm] (5,-1.5) rectangle (2.8,-0.9) node [below right] {Cyclohedra};

\end{tikzpicture}
\caption{Categories of nestohedra}\label{category}
\end{figure}
The first three are contained the class of \emph{graph associahedra} \cite{Carr-Devadoss2006}, that is, nestohedra whose building set~$\B = \B(G)$ consists of the connected induced subgraph of a finite simple graph~$G$.
For a finite simple graph~$G$, there is a graph invariant~$a(G)$, known as \emph{the $a$-number} of $G$, as introduced in \cite{Choi-Park2015}.
The Betti numbers of $X^\R_{\B(G)}$ is determined by the $a$-numbers of subgraphs of~$G$. 
See Section~\ref{graph_subsec} for more details.

A graph is said to be \emph{chordal} if every cycle containing four or more vertices has a chord, which is an edge that connects two nonadjacent vertices of the cycle.
According to \cite{Postnikov2008}, a finite simple graph $G$ is a \emph{chordal graph} with vertices labeled \emph{perfect elimination ordering} (refer Section~\ref{graph_subsec}) if and only if its building set $\B(G)$ is chordal.
For a chordal graph $G$ with a perfect elimination ordering, Theorem~\ref{main1} provides a combinatorial interpretation of the $a$-number $a(G)$ via $\B(G)$-permutations.
It is indeed beneficial because the computation of $a(G)$ was transformed into a counting problem, whereas the original definition of $a(G)$ is described recursively.
In particular, the results of Corollary~\ref{graph_corr} are equal to the $a$-numbers of \emph{complete graphs}, \emph{path graphs}, and \emph{star graphs}, which correspond permutohedra, associahedra, and stellohedra, respectively.

\begin{corollary}\label{graph_corr}
  \begin{enumerate}
    \item For a complete graph $K_{n+1}$ having $n+1$ vertices,
    $$
    \beta_k(X^{\R}_{\B(K_{n+1})}) = \sum_{I \in {[n+1] \choose 2k}}\#\text{alternating permutations on $I$}.
    $$
    \item For a path graph $G=P_{n+1}$ having $n+1$ vertices labeled consecutively as $1,\ldots,n+1$,
        $$
        \beta_k(X^\R_{\B(P_{n+1})}) = \sum_{I \in {[n+1] \choose 2k}} \prod_{1 \leq i \leq r_I} \# \text{$312$-avoiding alternating permutations on $I_i$},
        $$
        where the induced subgraph~$G \vert_I$ consists of the even order components $G \vert_{I_1}, \ldots, G \vert_{I_{r_I}}$.
    \item For a star graph $K_{1,n}$ having $n+1$ vertices with the central vertex labeled as $n+1$,
    $$
    \beta_k(X^\R_{\B(P_{n+1})}) = \sum_{I \in {[n] \choose 2k-1}}\#\{\text{alternating permutations on $I \cup \{n+1\}$, whose $1$st entry is $n+1$}\}.
    $$
  \end{enumerate}
\end{corollary}
In particular, when $n+1=2k$, $\beta_k(X^\R_{\B(P_{2k})})$, which is known to be the $k$th Catalan number~\cite{Choi-Park2015}, represents the number of $312$-avoiding alternating permutations on $[n+1]=[2k]$. 
This provides an alternative proof that the number of $312$-avoiding alternating permutations of length $2k$ is equal to the $k$th Catalan number~\cite{Stanley2015book-Catalan}.

Hochschild polytopes, including Stanley-Pitman polytopes, are chordal nestohedra, but most of them are not graph associahedra, as will be shown in Propositions~\ref{prop_Hoch} and~\ref{prop:Hoch_is_not_graph}.
For each pair~$(m,n)$ of non-negative integers, the $(m,n)$-Hochschild polytope $\Hoch(m,n)$ is defined as the nestohedron associated with the building set $\B$ on a set $[m+n] = \{1,2,\ldots,m+n\}$, where the building set $\B$ consists of $I \subset [m+n]$ that satisfy the following condition:
$$
    \left\vert I \right\vert \geq 2 \Rightarrow I \cap [m+1,m+n] \text{ is either } \emptyset \text{ or } [m+r,m+n] \text{ for } 1 \leq r \leq n,
$$
where $[i,j] \coloneqq \{ \ell \in \Z \colon i \leq \ell \leq j\}$ for integers $i\leq j$.
See Section~\ref{Hoch_subsec} for more details.  

Although the complex and real toric manifolds corresponding to a Hochschild polytope appear interesting, the authors are not aware of any existing topological and algebraic geometrical studies on them  as of the writing of this paper.
It would be appropriate to refer to such complex and real toric manifolds as the \emph{Hochschild} and \emph{real Hochschild variety}, respectively.
Due to our main theorem, we are able to compute the Betti numbers of real Hochschild varieties. 

For non-negative integers $s$ and $r$, consider the set of alternating permutations $(x_1x_2 \cdots x_{s+r})$ on $[s+r]$ such that
$$
s+r \geq x_{i} > x_{j} \geq s+1 \Rightarrow i <j,
$$
which is denoted by $\Alt_{\Hoch(s,r)}$.
\begin{corollary}\label{Hoch_corr}
  For each non-negative integers $m$ and $n$, the $k$th Betti numbers of a real Hochschild variety $X^\R_{\Hoch(m,n)}$ are
    $$
        \beta_{k}(X^\R_{\Hoch(m,n)}) = \sum_{s+r =2k}{m \choose s}\left\vert \Alt_{\Hoch(s,r)} \right\vert,
    $$
    where $s$ and $r$ are non-negative integers.
  
  Moreover, if $n \geq m+2$,
    $$
        \beta_{k}(X^\R_{\Hoch(m,n)}) = \beta_{k}(X^\R_{\Hoch(m,m+2)}),
    $$
    for all $k \geq 0$.
\end{corollary}

In the last section, we propose some interesting open problems related to the nested set complex and toric topology, motivated by the results of this paper.

\section{Preliminaries}\label{Preliminaries}

\subsection{Real toric manifolds associated with nestohedra}
Let us consider a complex $n$-dimensional toric manifold, denoted by~$X$, and its associated real toric manifold, denoted by~$X^\R$. 
By the fundamental theorem of toric geometry, $X$ corresponds to a non-singular complete fan~$\Sigma$ in $\R^n$.
Since $\Sigma$ is a simplicial fan, one can replace it with the pair $(K,\lambda)$ obtained by 
\begin{enumerate}
  \item a simplicial complex $K$ on the set $V$ of rays in $\Sigma$ such that a $(k-1)$-simplex of $K$ corresponds to a $k$-dimensional cone of $\Sigma$, and
  \item a linear map $\lambda$ from $V$ to $\Z^n$, called a \emph{characteristic map} of $X$, such that the image $\lambda(v)$ is the primitive vector in the direction of $v$ for each $v \in V$.
\end{enumerate}
Then $X$ is completely determined by $(K,\lambda)$.
In a similarly way, $X^\R$ is completely determined by a pair $(K,\lambda^\R)$, where $\lambda^{\R}$ is the map, called  a \emph{mod $2$ characteristic map} of $X^\R$, defined by the composition $V \xrightarrow{\lambda} \Z^n \xrightarrow{\text{mod} 2} \Z_2^n$.
One may regard $\lambda^\R$ by an $n \times m$ $\Z_2$-matrix, where each column is $\lambda^\R(v)$ for $v \in V$, and the row space of this matrix is denoted by $\row \lambda^\R$.
For each element~$\omega \in \row \lambda^\R$, $\omega$ can be regarded as a subset of $V$ via the standard correlation between the power set of $V$ and $\Z_2^m$.
Hence, $\omega$ is associated with the induced subcomplex $K_\omega$ of $K$ with respect to $\omega$.

Remark that the Betti numbers of $X^\R$ is determined by~$(K, \lambda^\R)$ in accordance with the following theorem.
\begin{theorem}[\cite{ST2012},\cite{Choi-Park2017_torsion}]\label{Choi-Park2017_torsion}
Let $X^\R$ be the real toric manifold associated with $(K,\lambda^\R)$.
Then the~$k$th rational Betti number is given by
$$
    \beta_k(X^\R) = \sum_{\omega \in \row  \lambda^\R} \widetilde{\beta}_{k-1} (K_{\omega}),
$$ 
where $K_{\omega}$ is the induced subcomplex of $K$ with respect to $\omega$ and $\widetilde{\beta}_{k-1}$ denotes the $(k-1)$th reduced Betti number.
\end{theorem}

For a finite set $S$ of natural numbers, a finite collection $\B$ of nonempty subsets of $S$ is a \emph{building set} on $S$ if
\begin{enumerate}
    \item $\{i\} \in \B$ for $i \in S$, and
    \item if $I,J \in \B$ and $I \cap J \neq \emptyset$, then $I \cup J \in \B$. 
\end{enumerate}
An inclusion-maximal subset in $\B$ is called a \emph{connected component} of $\B$, and we call $\B$ a \emph{connected building set} if $S \in \B$.
For a building set $\B$ on $S$ with $I \subset S$, the \emph{restricted building set} $\B \vert_I$ of $\B$ to $I$ is $\{J \in \B \colon J \subset I\}$.

Let $\B$ be a building set on $[n+1] = \{1,\ldots,n+1\}$.
The \emph{nestohedron}~$P_{\B}$ with respect to $\B$ is defined as the Minkowski sum of simplices
$$
    P_\B = \sum_{I \in \B}{\conv (I)},
$$
where $\conv (I)$ is the convex hull of the $i$th coordinate vectors in $\R^{n+1}$ for all $i \in I$.
A subset~$N \subseteq \B \setminus \{[n+1]\}$ is called a \emph{nested set} if it satisfies the following conditions:
\begin{enumerate}
  \item For each $I,J \in N$ one has either $I \subseteq J$, $J \subseteq I$, or $I \cap J = \emptyset$.
  \item For any collection of $k \geq 2$ disjoint subsets $I_1,\ldots,I_k \in N$, the union $I_1 \cup \cdots \cup I_k \notin \B$.
\end{enumerate}
The simplicial complex~$K_\B$ is defined as the collection of all nested sets of~$\B$ and is called the \emph{nested set complex} of~$\B$.
When~$\B$ is connected, according to \cite[Theorem7.4]{Postnikov2009}, $K_\B$ is dual to the boundary complex of~$P_\B$.

In addition, since the dimension of $P_\B$ is $n$, one can define maps~$\lambda_\B \colon \B \setminus \{[n+1]\} \to \Z^n$ and~$\lambda_\B^\R \colon \B \setminus \{[n+1]\} \to \Z_2^n$ by $\lambda_\B(I)$ is the primitive vector in the direction of the outward normal vector of the facet corresponding to $I$, and $\lambda_\B^\R(I) \equiv \lambda_\B(I) \pmod{2}$ for all $I\in \B \setminus \{[n+1]\}$.
More precisely, $\lambda^\R_{\B}$ is given by
$$
    \lambda^\R_{\B}(I) = \sum_{i \in I}e_i,
$$
where $e_i$ is the $i$th standard vector of $\Z^n_2$ for $1 \leq i \leq n$ and $e_{n+1} = \sum_{i=1}^{n}e_i$.

By \cite[Proposition~7.10]{Postnikov2009}, the pair $(K_\B, \lambda_\B)$ determines the toric manifold, denoted by~$X_\B$, and the pair $(K_\B, \lambda^\R_\B)$ determines the real toric manifold, denoted by~$X^\R_\B$.
The main purpose of this paper is to compute the rational Betti numbers of $X^\R_\B$ using Theorem~\ref{Choi-Park2017_torsion}.

Let $\omega_{i}$ be the element of $\row \lambda_{\B}^\R$ corresponding to the $i$th row of $\lambda_{\B}^\R$ as a matrix for $i=1, \ldots, n$.
Then, each element $\omega = \omega_{i_1} + \cdots + \omega_{i_k}$ of $\row \lambda_{\B}^\R$ can be identified with a subset $I_\omega \subset [n+1]$ that has an even cardinality as follows: 
\begin{equation}\label{awesome_build}
 I_\omega =
\begin{cases}
  \{i_1,\ldots,i_k\}, & \mbox{if } k \mbox{ is even} \\
  \{i_1,\ldots,i_k,n+1\}, & \mbox{otherwise}.
\end{cases} 
\end{equation}
For an even cardinality subset $I \subset [n+1]$, denote by $(K_{\B})_I$ the full subcomplex of~$K_{\B}$ induced by all vertices $J$, where $|J \cap I|$ is odd.
Then, one can easily see that $(K_{\B})_\omega$ is equal to $(K_{\B})_{I_\omega}$, and, conversely, for each an even subset $I$ of $[n+1]$, there is an element $\omega \in \row \lambda_\B^\R$ corresponding to~$I$.
Therefore, for a connected building set $\B$ on $[n+1]$, the Betti number formula in Theorem~\ref{Choi-Park2017_torsion} can be written as
\begin{equation}\label{Betti_comp}
  \beta_k(X_\B^\R) = \sum_{\substack{I \subset [n+1] \\ \left\vert I \right\vert \text{ is even}}} \widetilde{\beta}_{k-1} ((K_\B)_I).
\end{equation}

In this paper, we are particularly interested in some special class of building sets.
A building set $\B$ on $S$ is \emph{chordal} if for any $I = \{i_1< \cdots <i_r\} \in \B$, $\{i_s,\ldots,i_r\}$ is also an element of $\B$ for each $1 < s < r$.
It can be immediately observed that every restricted building set of a chordal building set is also chordal.
A nestohedron $P_\B$ is called a \emph{chordal nestohedron} if $\B$ is chordal.

\begin{example}
  A building set
  $$
  \B = \{\{1\} , \{2\} , \{3\} , \{4\} , \{1,4\} , \{3,4\} , \{1,3,4\} , \{2,3,4\} , \{1,2,3,4\}\}
  $$ is an example of connected chordal building set on $\{1,2,3,4\}$.
  The chordal nestohedron $P_\B$ is illustrated as in Figure~\ref{nestohedron_Figure}, and each facet of $P_\B$ is labeled by the corresponding element of $\B$.
  \begin{figure}[h]
    \centering
    \begin{tikzpicture}[scale=1.2]

\coordinate (A) at (0, 0);       
\coordinate (B) at (-0.2, 1);    
\coordinate (C) at (1, 1.5);      
\coordinate (D) at (0.7, 2.4);   
\coordinate (E) at (1.4, 2.6);    
\coordinate (F) at (1.5, 3.3);   
\coordinate (G) at (2, 3.7);      
\coordinate (H) at (2.7, 3.1);   
\coordinate (I) at (3.2, 3.5); 
\coordinate (J) at (3.3, 1.1);   
\coordinate (K) at (4.7, 2.1);  
\coordinate (L) at (5.5, -1);   

\draw[thick] (A) -- (B);
\draw[thick] (A) -- (C);
\draw[thick] (C) -- (D);
\draw[thick] (B) -- (D);
\draw[thick] (D) -- (F);
\draw[thick] (F) -- (H);
\draw[thick] (H) -- (J);
\draw[thick] (C) -- (J);
\draw[thick] (F) -- (G);
\draw[thick] (G) -- (I);
\draw[thick] (H) -- (I);
\draw[thick] (I) -- (K);
\draw[thick] (K) -- (L);
\draw[thick] (J) -- (L);
\draw[thick] (A) -- (L);

\draw[dashed] (B) -- (E);
\draw[dashed] (E) -- (G);
\draw[dashed] (E) -- (K);

\node at (0.4, 1.2) {\small{$\{1, 4\}$}};
\node at (2.3, 3.4) {\small{$\{3, 4\}$}}; 
\node at (1.8, 2.3) {$\{1, 3, 4\}$}; 
\node at (2.1, 0.4) {$\{1\}$}; 
\node at (4, 1.5) {$\{3\}$};
\node at (2.7, 1.4)  [color=gray] {\scriptsize{$\{2\}$}};
\node at (2.65, 2.9)  [color=gray] {\scriptsize{$\{2,3,4\}$}};
\node at (1.1, 2.55)  [color=gray] {\scriptsize{$\{4\}$}}; 
\end{tikzpicture}
    \caption{Chordal nestohedron $P_\B$}\label{nestohedron_Figure}
  \end{figure}
  
Let $K_\B$ be the face structure of $P_\B$ with reverse inclusion, and $\lambda_\B^\R$ its mod~$2$ characteristic map.
Then, $\lambda_\B^\R$ is represented by the~$3 \times 8$~matrix 
$$
\lambda_\B^\R =
\left(
\begin{array}{ccccccccc}
\{1\} & \{2\} & \{3\} & \{4\} & \{1,4\} & \{3,4\} & \{1,3,4\} & \{2,3,4\}\\
\hline
1 & 0 & 0 & 1 & 0 & 1 & 0 & 1\\
0 & 1 & 0 & 1 & 1 & 1 & 1 & 0\\
0 & 0 & 1 & 1 & 1 & 0 & 0 & 0\\
\end{array}
\right),
$$
where the first row indicates the labels of each facet.
Let $\omega_i$ be the $i$th row of $\row \lambda_\B^\R$ for $1 \leq i \leq 3$.
For instance, if we take $\omega=\omega_1 + \omega_2$, the corresponding even subset is $I_\omega = \{1,2\}$. 
One can see that both $(K_\B)_{\omega_1+\omega_2}$ and $(K_\B)_{\{1,2\}}$ are the full subcomplex induced by the set
$$
    \{ J \in \B \colon |J \cap \{1,2\}| \equiv 1 \pmod{2}\} = \{\{1\}, \{2\}, \{1,4\}, \{1,3,4\}, \{2,3,4\}\}.
$$
In this case, $(K_\B)_{I_\omega}$ is contractible, and, hence, $\tilde{\beta}_\ast ((K_\B)_{I_\omega})$ vanishes.
Therefore, it contribute nothing to the Betti number of $X_\B^\R$.
If we take $\omega = \omega_1+\omega_2+\omega_3$, the corresponding even subset is $I_\omega = \{1,2,3,4\}$.
One can see that both~$(K_\B)_{\omega}$ and~$(K_\B)_{I_\omega}$ are the full subcomplex induced by the set
$$
\{\{1\},\{2\},\{3\},\{4\},\{1,3,4\},\{2,3,4\}\}
$$ which is homotopy equivalent to a circle~$S^1$.
This follows from the fact that the simplicial complex~$K_\B$ is homeomorphic to the $2$-sphere~$S^2$, and the removal of the two non-adjacent $2$-faces~$\{1,4\}$ and~$\{3,4\}$ yields the subcomplex~$(K_\B)_{\omega}$, which is homeomorphic to~$S^2$ with two distinct points removed.  
It is well known that such a space is homotopy equivalent to~$S^1$.
Therefore, it contributes one to~$\beta_2 (X_\B^\R)$.
\end{example}

From the above example and many other similar examples, one can find interesting patterns in the computation of the reduced Betti numbers of $(K_\B)_I$;
if $|I|=2k$, then $(K_\B)_I$ is always homotopy equivalent to a bouquet of $(k-1)$-dimensional spheres, so it contributes only to~$\beta_{k}(X_\B^\R)$.
The purposes of this paper is to show that this phenomenon is not just a coincidence of specific cases, but rather a remarkable fact that holds generally for chordal building sets.
Refer Section~\ref{section3}.
In addition, we also show in Section~\ref{sec:alt_B_perm} that the number of spheres in each bouquet is obtainable by counting permutations satisfying conditions dependent on $\B$.

\subsection{Graph associahedra}\label{graph_subsec}
A building set $\B$ on $S$ is said to be \emph{graphical} if there is a finite simple graph $G$ on $S$ such that $I \in \B$ if and only if the induced subgraph $G \vert_I$ on $I$ is connected.
In this case, $\B$ is denoted by $\B(G)$.
Furthermore, every restricted building set $\B(G)\vert_I$ to $I$ is exactly $\B(G\vert_I)$, where $G \vert_I$ is the induced subgraph on $I$.
For a finite simple graph $G$, the nestohedron $P_{\B(G)}$ is called a \emph{graph associahedron}.
Graph associahedra have been widely studied in \cite{Carr-Devadoss2006}, \cite{Postnikov2009}, \cite{ToledanoLaredo2008}, and \cite{Zel2006}.

The Betti number of the real toric manifolds corresponding to a graphical building set is well-studied in \cite{Choi-Park2015}.
For a simple graph $G$ on $S$, the \emph{$a$-number} $a(G)$ of $G$ is defined as follows:
\begin{itemize}
  \item $a(\emptyset) = sa(\emptyset) = 1$.
  \item If $G$ is connected, 
  $$
  sa(G) = \begin{cases}
           - \underset{I \subsetneq S}{\sum}sa(G \vert_I), & \mbox{if $\left\vert S \right\vert$ is even,}  \\
            0, & \mbox{otherwise,}
          \end{cases}
  $$
  and $a(G) = \left\vert sa(G) \right\vert$.
  \item $a(G)$ is the product of $a(G_i)$ for $1 \leq i \leq l$, where $G_1,\ldots,G_l$ are the connected components of $G$.
\end{itemize}
By \cite[Theorem~1.1]{Choi-Park2015}, the $k$th Betti numbers of the associated real toric manifold~$X^\R_{\B(G)}$ is given by
$$
\beta_k(X^\R_{\B(G)}) = \sum_{I \subset {S \choose 2k}}a(G\vert_{I}).
$$

A finite simple graph $G$ on a set $S$ is \emph{chordal} if $G$ contains no induced $k$-cycle for $k \geq 4$. 
Refer~\cite{bondy08} for additional details.
The building set of a chordal graph is not necessary to be chordal.
However, if we give an appropriate labeling to nodes of a chordal graph, we can obtain a chordal building set.
A \emph{perfect elimination ordering} in $G$ is an ordering of the vertices of $G$ such that $G$ has no induced subgraph $G \vert_{\{i,j,k\}}$ with edges $(i,j),(i,k)$ but without the edge $(j,k)$, where $i < j < k$.
A graph $G$ is chordal if and only if $G$ admits a perfect elimination ordering~\cite{Fulkerson1965}.
According to \cite[Proposition~9.4]{Postnikov2008}, $G$ is a chordal graph labeled by a perfect elimination ordering if and only if the graphical building set $\B(G)$ is a chordal building set.

\begin{example}\label{graph_exam}
    Consider a path graph $P_4$ with $4$ vertices, which is a chordal graph.
    If $P_4$ is labeled by
      \begin{center}
        \begin{tikzpicture}[scale=1.2]
  \node (1) at (0,0) {$2$};
  \node (2) at (1,0) {$3$};
  \node (3) at (2,0) {$1$};
  \node(4) at (3,0) {$4$};
  
  \draw (1) -- (2);
  \draw (2) -- (3);
  \draw (3) -- (4);
\end{tikzpicture},
      \end{center}
then the graphical building set $\B(P_4)$ of $P_4$
    $$
    \{\{1\},\{2\},\{3\},\{4\}, \{2,3\},\{1,3\},\{1,4\}, \{1,2,3\},\{1,3,4\},\{1,2,3,4\}\}
    $$
    is not chordal.
    However, if we label $P_4$ by
      \begin{center}
        \begin{tikzpicture}[scale=1.2]
      \node (1) at (0,0) {$1$};
      \node (2) at (1,0) {$2$};
      \node (3) at (2,0) {$3$};
      \node(4) at (3,0) {$4$};
      
      \draw (1) -- (2);
      \draw (2) -- (3);
      \draw (3) -- (4);
        \end{tikzpicture}
      \end{center}
    that gives a perfect elimination ordering, then $\B(P_4)$
        $$
        \{\{1\},\{2\},\{3\},\{4\}, \{1,2\},\{2,3\},\{3,4\} , \{1,2,3\},\{2,3,4\}, \{1,2,3,4\} \}
        $$
      is chordal.
\end{example}

\subsection{Hochschild polytopes}\label{Hoch_subsec}
For non-negative integers $m$ and $n$, the \emph{$(m,n)$-Hochschild polytope}~$\Hoch(m,n)$ was introduced in \cite{Vincent-Daria2023}.
To describe its definition, we briefly introduce some notions following:
\begin{itemize}
  \item A finite sequence $\bS = \{(s_1,c_1),\ldots,(s_\ell,c_\ell)\}$ is called an \emph{$m$-lighted $n$-shade} if
\begin{enumerate}
  \item $s_i$ is a (possibly empty) tuple of positive integers and has total sum $n$,
  \item $c_1,\ldots,c_\ell$ are pairwise disjoint subsets of $[m]$ that their union $c_1 \cup c_2 \cup \cdots \cup c_\ell$ is $[m]$. 
  \item $s_i$ and $c_i$ cannot be both empty for each $i=1,\ldots,\ell$.
\end{enumerate}
  \item A \emph{rank} of $\bS = \{(s_1,c_1),\ldots,(s_\ell,c_\ell)\}$ is defined by $m - \ell  + \sum_{i = 1}^{\ell} \left\vert s_i\right\vert$, where~$\left\vert s_i \right\vert$ denotes the cardinality of~$s_i$.
  \item For $m$-lighted $n$-shades $\bS = \{(s_1,c_1),\ldots,(s_\ell,c_\ell)\}$ and $\bS'$, we denote by $\bS' \to \bS$ if $\bS$ and $\bS'$ satisfy one of the following:
  \begin{enumerate}
    \item there is $1 \leq i < \ell$ such that $\bS'$ is obtained from $\bS$ by exchange $(s_i,c_i),(s_{i+1},c_{i+1})$ to $(s_i+s_{i+1},c_i\cup c_{i+1})$, where $s_i = (s_i^1,\ldots,s_i^{p}), s_{i+1} = (s_{i+1}^1,\ldots,s_{i+1}^q)$, and $s_i+s_{i+1} = (s_i^1,\ldots,s_i^{p},s_{i+1}^1,\ldots,s_{i+1}^q)$.
    \item there is $1 \leq i \leq \ell$ such that $\bS'$ is obtained from $\bS$ by exchange $(s_i,c_i)$ to $(s'_i,c_i)$, where $s_i =(s_i^1,\ldots,s_i^{p})$, $s'_i = (s_i^1,\ldots,s_i^{r_1},s_i^{r_2},\ldots,s_i^p)$, and $s_i^r = s_i^{r_1}+s_i^{r_2}$ for some $1 \leq r \leq p$.
    \end{enumerate}
     \item $(\cP_{(m,n)},\leq)$ is the poset of $m$-lighted $n$-shades with the partial order $\leq$ defined by 
     $$\bS' \leq \bS \Leftrightarrow
      \begin{cases}
        \bS' \to \bS_1 \to \cdots \to \bS,\\
        \bS' = \bS.
      \end{cases}$$
\end{itemize}

For an $m$-lighted $n$-shade $\bS = \{(s_1,c_1),\ldots,(s_\ell,c_\ell)\}$, $\bS$ is visually represented as a vertical line with the tuples $s_1,\ldots,s_\ell$ on the left, and the sets of $c_1,\ldots,c_\ell$ on the right.
Since there is no confusion, we use the notation $a_1\cdots a_k$ for the tuple $(a_1,\ldots,a_k)$ or the set $\{a_1,\ldots,a_k\}$.
See Figure~\ref{figure1} for examples.
\begin{figure}
  \centering
\begin{tikzpicture}
    \draw (0,1.2) -- (0,1.8);
    \node[anchor=west] at (0,1.5) {12};
    \node[anchor=east] at (0,1.5) {1111};
    \node[anchor=north] at (0,0.5) {rank 5};
    \node[anchor=west] at (1,1.5) {$\to$};

    \draw (2.6,1) -- (2.6,2);
    \node[anchor=west] at (2.6,1.7) {2};
    \node[anchor=east] at (2.6,1.7) {111};
    \node[anchor=west] at (2.6,1.3) {1};
    \node[anchor=east] at (2.6,1.3) {1};
    \node[anchor=north] at (2.6,0.5) {rank 4};
    \node[anchor=west] at (3.5,1.5) {$\to$};

    \draw (5,1) -- (5,2);
    \node[anchor=west] at (5,1.7) {2};
    \node[anchor=east] at (5,1.7) {12};
    \node[anchor=west] at (5,1.3) {1};
    \node[anchor=east] at (5,1.3) {1};
    \node[anchor=north] at (5,0.5) {rank 3};
    \node[anchor=west] at (6,1.5) {$\to$};

    \draw (7.5,0.8) -- (7.5,2.2);
    \node[anchor=west] at (7.5,1.9) {2};
    \node[anchor=east] at (7.5,1.9) {12};
    \node[anchor=west] at (7.5,1.5) {1};
    \node[anchor=east] at (7.5,1.1) {1};
    \node[anchor=north] at (7.5,0.5) {rank 2};
    \node[anchor=west] at (8.5,1.5) {$\to$};

    \draw (10,0.7) -- (10,2.5);
    \node[anchor=west] at (10,2.2) {2};
    \node[anchor=east] at (10,2.2) {1};
    \node[anchor=east] at (10,1.8) {2};
    \node[anchor=west] at (10,1.4) {1};
    \node[anchor=east] at (10,1) {1};
    \node[anchor=north] at (10,0.5) {rank 1};
    \node[anchor=west] at (11,1.5) {$\to$};

    \draw (12.5,0.6) -- (12.5,2.7);
    \node[anchor=west] at (12.5,2.5) {2};
    \node[anchor=east] at (12.5,2.2) {1};
    \node[anchor=east] at (12.5,1.8) {2};
    \node[anchor=west] at (12.5,1.4) {1};
    \node[anchor=east] at (12.5,1) {1};
    \node[anchor=north] at (12.5,0.5) {rank 0};
\end{tikzpicture}
  \caption{A maximal chain of $2$-lighted $4$-shades}\label{figure1}
\end{figure}

\begin{remark}
  The poset of $1$-lighted $n$-shades gives a realization of the Hochschild lattice, which were introduced in \cite{Chapoton2020}.
Because of nice lattice properties of the Hochschild lattice, it has been studied in combinatorics \cite{Combe-Camille2021} and \cite{Muhle_Henri2022}.
\end{remark}

For non-negative integers $m$ and $n$, the $(m,n)$-Hochschild polytope $\Hoch(m,n)$ is defined as an abstract polytope whose face poset is anti-isomorphic to $(\cP_{(m,n)},\leq)$.
Refer \cite{Vincent-Daria2023} for the well-definedness of $\Hoch(m,n)$.
As brief examples, see Figure~\ref{figure2} for $(m,n)$-Hochschild polytopes and their face posets, represented by $m$-lighted $n$-shades for $m+n =3$.
\begin{figure}
  \centering
\begin{tikzpicture}[scale=0.52]
    \coordinate (A) at (-16,6);
    \coordinate (B) at (-13.45,4.5);
    \coordinate (C) at (-13.45,1.5);
    \coordinate (D) at (-16,0);
    \coordinate (E) at (-18.55,1.5);
    \coordinate (F) at (-18.55,4.5);

    \draw (A) -- (B) -- (C) -- (D) -- (E) -- (F) -- cycle;
    
    \draw (-16,6.1) -- (-16,7.3);
    \node[anchor=west] at (-16,7.1) {\tiny 1};
    \node[anchor=west] at (-16,6.7) {\tiny 2};
    \node[anchor=west] at (-16,6.3) {\tiny 3};
    
    \draw (-17.5,5.6) -- (-17.5,6.6);
    \node[anchor=west] at (-17.5,6.3) {\tiny 1};
    \node[anchor=west] at (-17.5,5.9) {\tiny 23};
    
    \draw (-14.5,5.6) -- (-14.5,6.6);
    \node[anchor=west] at (-14.5,6.3) {\tiny 12};
    \node[anchor=west] at (-14.5,5.9) {\tiny 3};
    
    \draw (-13.1,4.6) -- (-13.1,5.8);
    \node[anchor=west] at (-13.1,5.6) {\tiny 2};
    \node[anchor=west] at (-13.1,5.2) {\tiny 1};
    \node[anchor=west] at (-13.1,4.8) {\tiny 3};
    
    \draw (-13.1,2.5) -- (-13.1,3.5);
    \node[anchor=west] at (-13.1,3.2) {\tiny 2};
    \node[anchor=west] at (-13.1,2.8) {\tiny 13};
    
    \draw (-19.5,2.5) -- (-19.5,3.5);
    \node[anchor=west] at (-19.5,3.2) {\tiny 13};
    \node[anchor=west] at (-19.5,2.8) {\tiny 2};
    
    \draw (-13.1,1.4) -- (-13.1,0.2);
    \node[anchor=west] at (-13.1,0.4) {\tiny 1};
    \node[anchor=west] at (-13.1,0.8) {\tiny 3};
    \node[anchor=west] at (-13.1,1.2) {\tiny 2};
    
    \draw (-19.2,1.4) -- (-19.2,0.2);
    \node[anchor=west] at (-19.2,0.4) {\tiny 2};
    \node[anchor=west] at (-19.2,0.8) {\tiny 1};
    \node[anchor=west] at (-19.2,1.2) {\tiny 3};
    
    \draw (-16,-1.4) -- (-16,-0.2);
    \node[anchor=west] at (-16,-1.2) {\tiny 1};
    \node[anchor=west] at (-16,-0.8) {\tiny 2};
    \node[anchor=west] at (-16,-0.4) {\tiny 3};
    
    \draw (-19.2,4.6) -- (-19.2,5.8);
    \node[anchor=west] at (-19.2,5.6) {\tiny 1};
    \node[anchor=west] at (-19.2,5.2) {\tiny 3};
    \node[anchor=west] at (-19.2,4.8) {\tiny 2};
    
    \draw (-17.5,-0.4) -- (-17.5,0.6);
    \node[anchor=west] at (-17.5,0.3) {\tiny 3};
    \node[anchor=west] at (-17.5,-0.1) {\tiny 12};
    
    \draw (-14.5,-0.4) -- (-14.5,0.6);
    \node[anchor=west] at (-14.5,0.3) {\tiny 23};
    \node[anchor=west] at (-14.5,-0.1) {\tiny 1};
    
    \draw (-16,2.6) -- (-16,3.4);
    \node[anchor=west] at (-16,3) {\tiny 123};

    \coordinate (A) at (-8,6);
    \coordinate (B) at (-5.45,4.5);
    \coordinate (C) at (-5.45,1.5);
    \coordinate (D) at (-8,0);
    \coordinate (E) at (-10.55,1.5);
    \coordinate (F) at (-10.55,4.5);

    \draw (A) -- (B) -- (C) -- (D) -- (E) -- (F) -- cycle;
    
    \draw (-8,6.1) -- (-8,7.3);
    \node[anchor=west] at (-8,7.1) {\tiny 1};
    \node[anchor=west] at (-8,6.7) {\tiny 2};
    \node[anchor=east] at (-8,6.3) {\tiny 1};
    
    \draw (-9.5,5.6) -- (-9.5,6.6);
    \node[anchor=west] at (-9.5,6.3) {\tiny 1};
    \node[anchor=west] at (-9.5,5.9) {\tiny 2};
    \node[anchor=east] at (-9.5,5.9) {\tiny 1};
    
    \draw (-6.5,5.6) -- (-6.5,6.6);
    \node[anchor=west] at (-6.5,6.3) {\tiny 12};
    \node[anchor=east] at (-6.5,5.9) {\tiny 1};
    
    \draw (-4.8,4.6) -- (-4.8,5.8);
    \node[anchor=west] at (-4.8,5.6) {\tiny 2};
    \node[anchor=west] at (-4.8,5.2) {\tiny 1};
    \node[anchor=east] at (-4.8,4.8) {\tiny 1};
    
    \draw (-4.8,2.5) -- (-4.8,3.5);
    \node[anchor=west] at (-4.8,3.2) {\tiny 2};
    \node[anchor=west] at (-4.8,2.8) {\tiny 1};
    \node[anchor=east] at (-4.8,2.8) {\tiny 1};
    
    \draw (-11.2,2.5) -- (-11.2,3.5);
    \node[anchor=west] at (-11.2,3.2) {\tiny 1};
    \node[anchor=east] at (-11.2,3.2) {\tiny 1};
    \node[anchor=west] at (-11.2,2.8) {\tiny 2};
    
    \draw (-4.8,1.4) -- (-4.8,0.2);
    \node[anchor=west] at (-4.8,0.4) {\tiny 1};
    \node[anchor=east] at (-4.8,0.8) {\tiny 1};
    \node[anchor=west] at (-4.8,1.2) {\tiny 2};
    
    \draw (-11.2,1.4) -- (-11.2,0.2);
    \node[anchor=west] at (-11.2,0.4) {\tiny 2};
    \node[anchor=west] at (-11.2,0.8) {\tiny 1};
    \node[anchor=east] at (-11.2,1.2) {\tiny 1};
    
    \draw (-8,-1.4) -- (-8,-0.2);
    \node[anchor=west] at (-8,-1.2) {\tiny 1};
    \node[anchor=west] at (-8,-0.8) {\tiny 2};
    \node[anchor=east] at (-8,-0.4) {\tiny 1};
    
    \draw (-11.2,4.6) -- (-11.2,5.8);
    \node[anchor=west] at (-11.2,5.6) {\tiny 1};
    \node[anchor=east] at (-11.2,5.2) {\tiny 1};
    \node[anchor=west] at (-11.2,4.8) {\tiny 2};
    
    \draw (-9.5,-0.4) -- (-9.5,0.6);
    \node[anchor=east] at (-9.5,0.3) {\tiny 1};
    \node[anchor=west] at (-9.5,-0.1) {\tiny 12};
    
    \draw (-6.5,-0.4) -- (-6.5,0.6);
    \node[anchor=west] at (-6.5,0.3) {\tiny 2};
    \node[anchor=east] at (-6.5,0.3) {\tiny 1};
    \node[anchor=west] at (-6.5,-0.1) {\tiny 1};
    
    \draw (-8,2.6) -- (-8,3.4);
    \node[anchor=west] at (-8,3) {\tiny 12};
    \node[anchor=east] at (-8,3) {\tiny 1};

    \coordinate (A) at (0,6);
    \coordinate (B) at (2.55,4.5);
    \coordinate (C) at (2.55,1.5);
    \coordinate (D) at (-2.55,-1);
    \coordinate (E) at (-2.55,4.5);

    \draw (A) -- (B) -- (C) -- (D) -- (E) -- cycle;
    
    \draw (0,6.1) -- (0,7.3);
    \node[anchor=west] at (0,7.1) {\tiny 1};
    \node[anchor=east] at (0,6.7) {\tiny 1};
    \node[anchor=east] at (0,6.3) {\tiny 1};
    
    \draw (-1.5,5.3) -- (-1.5,6.3);
    \node[anchor=west] at (-1.5,6.1) {\tiny 1};
    \node[anchor=east] at (-1.5,5.7) {\tiny 11};
    
    \draw (1.5,5.5) -- (1.5,6.5);
    \node[anchor=east] at (1.5,6.2) {\tiny 1};
    \node[anchor=west] at (1.5,6.2) {\tiny 1};
    \node[anchor=east] at (1.5,5.8) {\tiny 1};
    
    \draw (3.1,4.6) -- (3.1,5.8);
    \node[anchor=east] at (3.1,5.6) {\tiny 1};
    \node[anchor=west] at (3.1,5.2) {\tiny 1};
    \node[anchor=east] at (3.1,4.8) {\tiny 1};
    
    \draw (3.3,2.5) -- (3.3,3.5);
    \node[anchor=east] at (3.3,3.2) {\tiny 1};
    \node[anchor=west] at (3.3,2.8) {\tiny 1};
    \node[anchor=east] at (3.3,2.8) {\tiny 1};
    
    \draw (3,1.4) -- (3,0.2);
    \node[anchor=west] at (3,0.4) {\tiny 1};
    \node[anchor=east] at (3,0.8) {\tiny 1};
    \node[anchor=east] at (3,1.2) {\tiny 1};
    
    \draw (0.2,-1) -- (0.2,0);
    \node[anchor=west] at (0.2,-0.7) {\tiny 1};
    \node[anchor=east] at (0.2,-0.3) {\tiny 11};
    
    \draw (-3,-1.1) -- (-3,-2);
    \node[anchor=west] at (-3,-1.7) {\tiny 1};
    \node[anchor=east] at (-3,-1.4) {\tiny 2};
    
    \draw (-3.2,4) -- (-3.2,4.8);
    \node[anchor=west] at (-3.2,4.6) {\tiny 1};
    \node[anchor=east] at (-3.2,4.2) {\tiny 2};
    
    \draw (-0.5,2.6) -- (-0.5,3.4);
    \node[anchor=west] at (-0.5,3) {\tiny 1};
    \node[anchor=east] at (-0.5,3) {\tiny 11};
    
    \draw (-3.2,1.6) -- (-3.2,2.4);
    \node[anchor=west] at (-3.2,2) {\tiny 1};
    \node[anchor=east] at (-3.2,2) {\tiny 2};
    
    \coordinate (a) at (8,6);
    \coordinate (b) at (10,3);
    \coordinate (c) at (5.45,-1);
    \coordinate (d) at (5.45,4.5);

    \draw (a) -- (b) -- (c) -- (d) -- cycle;

    \draw (8,6.1) -- (8,7.3);
    \node[anchor=east] at (8,7.1) {\tiny 1};
    \node[anchor=east] at (8,6.7) {\tiny 1};
    \node[anchor=east] at (8,6.3) {\tiny 1};

    \draw (6.5,5.3) -- (6.5,6.3);
    \node[anchor=east] at (6.5,6.1) {\tiny 1};
    \node[anchor=east] at (6.5,5.7) {\tiny 11};

    \draw (10.1,4.5) -- (10.1,5.5);
    \node[anchor=east] at (10.1,5.2) {\tiny 11};
    \node[anchor=east] at (10.1,4.8) {\tiny 1};

    \draw (10.8,2.5) -- (10.8,3.5);
    \node[anchor=east] at (10.8,3.2) {\tiny 2};
    \node[anchor=east] at (10.8,2.8) {\tiny 1};

    \draw (9.2,0.3) -- (9.2,1);
    \node[anchor=east] at (9.2,0.7) {\tiny 21};

    \draw (5.1,-0.9) -- (5.1,-1.7);
    \node[anchor=east] at (5.1,-1.3) {\tiny 3};

    \draw (5,4) -- (5,4.8);
    \node[anchor=east] at (5,4.6) {\tiny 1};
    \node[anchor=east] at (5,4.2) {\tiny 2};

    \draw (8,2.6) -- (8,3.4);
    \node[anchor=east] at (8,3) {\tiny 111};

    \draw (5,1.6) -- (5,2.4);
    \node[anchor=east] at (5,2) {\tiny 12};
\end{tikzpicture}
  \caption{$\Hoch(m,n)$ for $(m,n) = (3,0), (2,1), (1,2)$ and $(0,3)$}\label{figure2}
\end{figure}

Let $\B_{m,n}$ be the set consisting of non-empty subsets $I$ of $[m+n]$ such that
\begin{equation}\label{Hoch_build}
    \left\vert I \right\vert \geq 2 \Rightarrow I \cap [m+1,m+n] \text{ is either } \emptyset \text{ or } [m+r,m+n] \text{ for } 1 \leq r \leq n,
\end{equation}
where $[i,j]$ denotes the set of all integers $\ell$ such that $i \leq \ell \leq j$.
It is clear that $\B_{m,n}$ is a connected chordal building set on $[m+n]$.

\begin{proposition}\label{prop_Hoch}
    For non-negative integers $m$ and $n$, the nestohedron $P_{\B_{m,n}}$ has a same combinatorial structure as the $(m,n)$-Hochschild polytope $\Hoch(m,n)$.    
\end{proposition}

\begin{proof}
    Consider an $m$-lighted $n$-shade
    $$
    \bS = \{(s_1,c_1),\ldots,(s_\ell = (s^1_\ell,\ldots,s^p_\ell),c_\ell)\},
    $$
    whose rank is not $m+n-1$.
    When there exists an index $1 \leq i \leq p$ such that $s^i_\ell$ is at least $2$, we denote the maximal such index by $1 \leq q \leq p$.
    Define a subset
    $$
      B_\bS =  \begin{cases}
      c_\ell, & \mbox{if $s_\ell = \emptyset$}, \\
        c_\ell \cup \{m+n-i \colon 0 \leq i < \left\vert s_\ell\right\vert\}, & \mbox{if $s_\ell \neq \emptyset$ and all entries of $s_\ell$ are $1$}, \\
        \{m+n-(p+1-q)\}, & \mbox{otherwise,}
      \end{cases}
    $$
    of $[m+n]$, and an $m$-lighted $n$-shade $\varphi(\bS)$ obtained from $\bS$ as follows:
    \begin{enumerate}
      \item By exchange $(s_{\ell-1},c_{\ell-1}),(s_{\ell},c_{\ell})$ to $(s_{\ell-1}+s_{\ell},c_{\ell-1}\cup c_{\ell})$ if all entries of $s_\ell$ are $1$,
      \item By exchange $s_\ell = (\ldots,s^q_\ell,\ldots)$ to $(\ldots,s_{\ell}^q-1,1,\ldots)$, otherwise.
    \end{enumerate}
    For any non-negative integer $i$, $\varphi^i(\bS)$ is denoted by the composition
    $$
    \varphi^i(\bS)=\begin{cases}
      \bS & \mbox{if } i = 0 \\
      \underbrace{\varphi \cdot \varphi \cdot \cdots \cdot \varphi}_\text{$i$ times}(\bS), & \mbox{otherwise}.
    \end{cases}
    $$
    
    For an $m$-lighted $n$-shade $\bS$ whose rank is $r$, define
    $$
    \Phi(\bS) = \{B_{\varphi^i(\bS)} \colon 0 \leq i \leq m+n-2-r\}.
    $$
    Indeed, $\Phi$ induces a well-defined order isomorphism from $(\cP_{(m,n)}\setminus \{[m+n]\},\leq)$ to the nested set complex $K_{\B_{m,n}}$.
\end{proof}

\begin{example}
Let $\cC$ be a chain of $2$-lighted $4$-shades as follows:\\
\begin{center}
\begin{tikzpicture}
    \node[anchor=east] at (7,1.5) {$\bS =$};
    
    \draw (7.5,1) -- (7.5,2);
    \node[anchor=west] at (7.5,1.7) {2};
    \node[anchor=east] at (7.5,1.7) {12};
    \node[anchor=west] at (7.5,1.3) {1};
    \node[anchor=east] at (7.5,1.3) {1};
    \node[anchor=north] at (7.5,0.5) {rank 3};
    \node[anchor=west] at (8.5,1.5) {$\leq$};
    \draw (10,0.7) -- (10,2.5);
    \node[anchor=west] at (10,2.2) {2};
    \node[anchor=east] at (10,2.2) {1};
    \node[anchor=east] at (10,1.8) {2};
    \node[anchor=west] at (10,1.4) {1};
    \node[anchor=east] at (10,1) {1};
    \node[anchor=north] at (10,0.5) {rank 1};
    \node[anchor=west] at (11,1.5) {$\leq$};
    \draw (12.5,0.6) -- (12.5,2.7);
    \node[anchor=west] at (12.5,2.5) {2};
    \node[anchor=east] at (12.5,2.2) {1};
    \node[anchor=east] at (12.5,1.8) {2};
    \node[anchor=west] at (12.5,1.4) {1};
    \node[anchor=east] at (12.5,1) {1};
    \node[anchor=north] at (12.5,0.5) {rank 0};
\end{tikzpicture}.
\end{center}
We obtain~$B_\bS = \{1,6\}.$
Since
\begin{center}
\begin{tikzpicture}
    \node[anchor=east] at (6.6,1.5) {$\varphi(\bS) =$};
    
    \draw (7.5,1.2) -- (7.5,1.8);
    \node[anchor=west] at (7.5,1.5) {12};
    \node[anchor=east] at (7.5,1.5) {121};
\end{tikzpicture},
\end{center}
we have~$B_{\varphi(\bS)} = \{4\}$.
The remaining cases follow in the same way.
The chain $\cC$ corresponds to the chain of nested sets
  $$
  \{\{1,6\},\{4\}\} \subset \{\{6\},\{1,6\},\{4\},\{1,4,5,6\}\} \subset \{\{6\},\{1,6\},\{4\},\{1,4,5,6\},\{1,3,4,5,6\}\}
  $$
of $\B_{m,n}\setminus \{[m+n]\}$ for $m = 2, n = 4$.
\end{example}

For each pair $(m,n)$ of non-negative integers, since the $(m,n)$-Hochschild polytope can be realized as a nestohedron $P_{\B_{m,n}}$, $\Hoch(m,n)$ induces the complex and real toric manifold, and they are denoted by $X_{\Hoch(m,n)}$ and $X_{\Hoch(m,n)}^\R$, respectively.
Moreover, we call such complex and real toric manifolds \emph{Hochschild} and \emph{real Hochschild variety}, respectively.

In particular, $\Hoch(0,n)$ is known as a Stanley-Pitman polytope \cite{Stanley2015book-Catalan}.
It is easy to verify that $\Hoch(m,n)$ forms a graph associahedron when $n \leq 2$, and specifically, it forms a permutohedron when $n \leq 1$.

\begin{proposition} \label{prop:Hoch_is_not_graph}
   For $n \geq 3$, a building set $\B_{m,n}$ is not a graphical building set.
\end{proposition}
\begin{proof}
  Assume that $\B := \B_{m,n}$ is graphical when $n \geq 3$.
Choose a finite simple graph $G$ with~$\B =\B(G)$.
Let $I = \{m+1,\ldots,m+n\}$.
Then the restricted building set $\B \vert_I$ should be a building set with respect to an subgraph~$G\vert_I$ of $G$ induced by $I$.
Since $I \in \B \vert_I$, the subgraph~$G \vert_I$ is connected, and~$G \vert_I$ consists of at least $n-1$ edges; that is, $\B$ must contain at least $n-1$ elements of cardinality~$2$.
However, this contradicts that $\B \vert_I$ has a unique element~$\{m+n-1,m+n\}$ of cardinality $2$.
\end{proof}

\section{EL-shellability of the nested set complex}\label{section3}

For a graded poset $\cP$, let us consider the \emph{order complex} $\Delta(\cP)$, that is, $\Delta(\cP)$ is the simplicial complex whose simplices are finite chains of $\cP$.
We assume that a graded poset $\cP$ is \emph{bounded}, which means that $\cP$ has both a least and a greatest element.

Let $\bE(\cP)$ denote the set of edges of the Hasse diagram of $\cP$, which is the set of pairs of elements in $\cP$ such that one covers the other.
An \emph{edge labeling} $\mu$ of $\cP$ is a map from $\bE(\cP)$ to a poset $(\cQ,\succeq)$.
For an edge labeling $\mu$ and for each chain $c = (z_0, z_1, \dots ,z_k)$ of $\cP$ such that $z_{i}$ covers $z_{i-1}$ for all $i=1, \ldots, k$, we denote by~$\mu(c)$ the $k$-tuple
  $$
  \mu(c)=(\mu(z_{0},z_{1}),\mu(z_1,z_2),\ldots,\mu(z_{k-1},z_k)).
  $$
An edge labeling $\mu \colon \bE(\cP) \to \cQ$ is \emph{reversed edge lexicographical} if for each interval $[x,y]$ in $\cP$,
  \begin{enumerate}
    \item there is a unique maximal chain $c = (z_0, z_1, \dots ,z_k)$ of $[x,y]$ such that $\mu(c)$ is \emph{decreasing}, \ie
    $$
    \mu(z_0,z_1) \succeq \mu(z_1,z_2) \succeq \cdots \succeq \mu(z_{k-1},z_k), \text{ and}
    $$
    \item for all other maximal chains $c' = (z_0', \dots, z_k')$ of $[x,y]$, $\mu(c)$ lexicographically succeeds $\mu(c')$, \ie $\mu(z_{i-1},z_i) \succ \mu(z_{i-1}',z_i')$ for $i = \min \{1 \leq j \leq k \colon \mu(z_{j-1},z_j) \neq \mu(z_{j-1}',z_j')\}$.
    \end{enumerate}   
A bounded graded poset $\cP$ is said to be \emph{EL-shellable} if $\cP$ admits an reversed edge lexicographical edge labeling. 
\footnote{The original definition of EL-shellability in \cite{Bjorner-Wachs1982} requires for $\cP$ to admit an edge labeling $\mu$ that is edge lexicographical. However, in this case, $-\mu$ is a reversed edge lexicographical edge labeling, making our definition essentially equivalent to the original one.}

 In this section, unless otherwise stated, we consider a set of natural numbers with an even cardinality $2k$.
For a building set~$\B$ on $S$, define the set~$\widehat{\cP}_\B$ of subsets~$I$ of~$S$ such that 
$$
    \widehat{\cP}_\B \coloneqq \{ I \subset S \colon \text{$\B \vert_I$ has no connected components of odd order} \}.
$$
Note that $(\widehat{\cP}_\B, \subseteq)$ is a bounded graded poset, and each edge $(I_1,I_2)$ of $\bE(\widehat{\cP}_\B)$ consists of two subsets satisfying $I_1 \subset I_2$ and $\left\vert I_2 \right\vert = \left\vert I_1 \right\vert + 2$.
We note that $I_2 \setminus I_1$ should be contained in the same connected component, denoted by $\fC(I_2;I_1)$, of the restricted building set $\B \vert_{I_2}$.

Now, we introduce some notions to define an edge labeling on $\widehat{\cP}_\B$ that will be reversed edge lexicographical.
We define
$$
\Omega = \{(x,y) \in \Z \times \Z \colon x \geq y\}.
$$

\begin{definition}
    Let $\cR_1$ and $\cR_2$ be the subsets of $\Omega \times \Omega$ defined by
    \begin{align*}
  \cR_1 &= \big\{\big((x,y),(x',y')\big) \in \Omega \times \Omega \colon y  \geq x' \big\}, \\
  \cR_2 &=\big\{\big((x,y),(x',y')\big) \in \Omega \times \Omega \colon x = x' \text{ and } y \geq y' \big\}.
\end{align*}
    For each pair of elements $\alpha$ and $\beta$ of the set $\Omega$, define the relation~$\geq$ on $\Omega$ by $\alpha \geq \beta$ if the pair $(\alpha,\beta)$ is an element of either $\cR_1$ or $\cR_2$.

\end{definition}

\begin{proposition}
  The relation $\geq$ is a partial order on $\Omega$.
\end{proposition}

\begin{proof}
    To begin with, for each $\alpha \in \Omega$, since $(\alpha,\alpha) \in \cR_2$, the reflexivity holds.
    
    For each pair of elements $\alpha =(\alpha_1,\alpha_2)$ and $\beta = (\beta_1,\beta_2)$ of $\Omega$, if $\alpha \geq \beta$ and $\beta \geq \alpha$, then
    $$
    \begin{cases}
      \alpha_1 \geq \alpha_2 \geq \beta_1 \geq \beta_2 \geq \alpha_1 \geq \alpha_2, & \mbox{if } (\alpha,\beta),(\beta,\alpha) \in \cR_1, \\
      \alpha_1 \geq \alpha_2 \geq \beta_1 = \alpha_1 \geq \beta_2 \geq \alpha_2 , & \mbox{if }  (\alpha,\beta) \in \cR_1, (\beta,\alpha) \in \cR_2,\\
      \alpha_1 = \beta_1 \geq \alpha_2 \geq \beta_2 \geq \alpha_2, & \mbox{if } (\alpha,\beta),(\beta,\alpha) \in \cR_2.
    \end{cases}
    $$
    Consequently, in all cases, we conclude that $\alpha =\beta$ which confirms the antisymmetry.
    
    Let $\alpha =(\alpha_1,\alpha_2)$, $\beta = (\beta_1,\beta_2)$, and $\gamma = (\gamma_1,\gamma_2)$ be elements of $\Omega$.
    Assume that $\alpha \geq \beta$ and~$\beta \geq \gamma$.
    Then, we have the following four cases.
    $$
    \begin{cases}
      \alpha_1 \geq \alpha_2 \geq \beta_1 \geq \beta_2 \geq \gamma_1 \geq \gamma_2, & \mbox{if } (\alpha,\beta),(\beta,\gamma) \in \cR_1, \\
      \alpha_1 \geq \alpha_2 \geq \beta_1 = \gamma_1 \geq \beta_2 \geq \gamma_2 , & \mbox{if }  (\alpha,\beta) \in \cR_1, (\beta,\gamma) \in \cR_2,\\
      \alpha_1 =\beta_1 \geq \alpha_2 \geq \beta_2 \geq \gamma_1 \geq \gamma_2 , & \mbox{if }  (\alpha,\beta) \in \cR_2, (\beta,\gamma) \in \cR_1,\\
      \alpha_1 = \beta_1 = \gamma_1 \geq \alpha_2 \geq \beta_2 \geq \gamma_2, & \mbox{if } (\alpha,\beta),(\beta,\gamma) \in \cR_2. 
    \end{cases}
    $$
In the first three cases, $(\alpha,\gamma) \in \cR_1$, and, in the final case, $(\alpha,\gamma) \in \cR_2$.
Therefore, $\alpha \geq \gamma$, which proves the transitivity.
\end{proof}

Let $\succeq$ be the lexicographic order on the product $(\Z,\geq) \times (\Omega,\geq)$, that is,
$$
(x,\alpha) \succeq (y,\beta) \text{ if and only if } \begin{cases}
x > y  , \text{ or} \\
x = y, \alpha \geq \beta.
\end{cases}
$$
Define an edge labeling $\mu_\B \colon \bE(\widehat{\cP}_\B) \to (\Z \times \Omega , \succeq)$ by
$$
    \mu_\B(I_1,I_2) = (\max \fC(I_2;I_1) ,(\max I_2 \setminus I_1,\min I_2 \setminus I_1))
$$
for each $(I_1,I_2) \in \bE(\widehat{\cP}_\B)$.

\begin{example}
Consider the building set $\B = \B_{2,4}$.
The Hasse diagram $\bE(\widehat{\cP}_\B)$ of~$\B$ and the edge labeling $\mu_\B$ are illustrated in Figure~\ref{example_REL}.
\begin{figure}[h!]
  \centering
\begin{tikzpicture}[-,>=stealth',shorten >=1.5pt,auto,node distance=3.5cm,
                    thick,main node/.style={draw,font=\sffamily\normalsize}]

  \node (empty) {$\emptyset$};
  \node (62) [above right of=empty] {$\{2,6\}$};
  \node (65) [right of=62] {$\{5,6\}$};
  \node (21) [above left of=empty] {$\{1,2\}$};
  \node (61) [left of=21] {$\{1,6\}$};
  \node (6542) [above of=62] {$\{2,4,5,6\}$};
  \node (6543) [above of=65] {$\{3,4,5,6\}$};
  \node (6541) [above of=61] {$\{1,4,5,6\}$};
  \node (6521) [above of=21] {$\{1,2,5,6\}$};
  \node (654321) [above right of=6521] {$\{1,2,3,4,5,6\}$};

  \path[every node/.style={font=\sffamily\footnotesize}]
    (empty) edge node [below right,pos=0.3] {$(6,(6,5))$} (65)
     edge node [left, pos=0.6] {$(6,(6,2))$} (62)
     edge node [below left , pos=0.4] {$(6,(6,1))$} (61)
     edge node [left, pos=0.8] {$(2,(2,1))$} (21)
    (61) edge node [left] {$(6,(5,4))$} (6541)
    edge node [left, pos=0.6] {$(6,(5,2))$} (6521)
    (65) edge node [right, pos= 0.7] {$(6,(4,2))$} (6542)
    edge node [right] {$(6,(4,3))$} (6543)
    edge node [below, pos=0.65] {$(6,(4,1))$} (6541)
    edge node [above right, pos=0.86] {$(6,(2,1))$} (6521)
    (21) edge node [left, pos=0.2] {$(6,(6,5))$} (6521)
    (62) edge node [below left , pos=0.2] {$(6,(5,1))$} (6521)
    edge node [left, pos=0.86] {$(6,(5,4))$} (6542)
    (6542) edge node [right, pos=0.3] {$(6,(3,1))$} (654321)
    (6543) edge node [above right] {$(6,(2,1))$} (654321)
    (6541) edge node [left] {$(6,(3,2))$} (654321)
    (6521) edge node [right] {$(6,(4,3))$} (654321);
\end{tikzpicture}
\caption{$\bE(\widehat{\cP}_{\B_{2,4}})$ and $\mu_{\B_{2,4}}$}\label{example_REL}
\end{figure}

For an interval $[\emptyset, \{1,2,3,4,5,6\}]$, the chain
$$
c = (\emptyset,\{5,6\},\{3,4,5,6\},\{1,2,3,4,5,6\})
$$
of $\widehat{\cP}_\B$ is the unique maximal chain such that $\mu_\B(c)$ is decreasing and, for all other maximal chains $c'$ of $\widehat{\cP}_\B$, $\mu_\B(c)$ lexicographically succeeds $\mu_\B(c')$.
For other intervals, one can find such a unique maximal chain.
Therefore, $\mu_\B$ forms a reversed lexicographical edge labeling, which confirms that $\widehat{\cP}_{\B_{2,4}}$ is EL-shellable.
\end{example}

The following lemma represents one of the central and most crucial ideas of this research.

\begin{lemma}\label{shellable}
    If $\B$ is a chordal building set, then $(\widehat{\cP}_\B,\subseteq)$ is EL-shellable.
\end{lemma}

\begin{proof}
  Take any interval $[I,J]$ of $\widehat{\cP}_\B$, where $I \subsetneq J$.
  Let $\cC_1,\ldots,\cC_r$ be all connected components of the restricted building set $\B \vert_{J}$.
  We assume that $\max \cC_1 > \max \cC_2 > \cdots > \max \cC_r$.
  Since both $\B \vert_I$ and $\B \vert_J$ have no connected component of odd order, each cardinality of $\cC_i \setminus I$, denoted by $2k_i$, is even.
  Then the cardinality of $J \setminus I$, denoted by $2N$, is $2k_1 + \cdots + 2k_r$.
  For each $1 \leq i \leq r$, we express
  $$
  \cC_i \setminus I = \{\alpha_{2k_0+\cdots+2k_{i-1}+1},  \alpha_{2k_0+\cdots+2k_{i-1}+2}, \ldots, \alpha_{2k_0+\cdots+2k_{i-1}+2k_i}\},
  $$
  where $k_0 = 0$ and $\alpha_{2k_0+\cdots+2k_{i-1}+s} > \alpha_{2k_0+\cdots+2k_{i-1}+t}$ if $1 \leq s < t \leq 2k_i$.
  
  Let $1 \leq \ell \leq N$ be given.
  Define $I_\ell = I \cup \{\alpha_1,\alpha_2,\ldots,\alpha_{2\ell}\}$ and the index $1 \leq s_\ell \leq r$ such that $\fC(I_\ell;I_{\ell-1}) \subset \cC_{s_\ell}$.
  Then one can express
  \begin{equation}\label{I_l}
  I_{\ell} = I \cup \{\alpha_1,\ldots,\alpha_{2\ell}\}
   = I \cup \cC_0 \cup \cdots \cup \cC_{s_\ell-1} \cup \{\alpha \in \cC_{s_\ell} \colon \alpha \geq \alpha_{2\ell}\},
  \end{equation}
  where $\cC_0 =\emptyset$.
  By the chordality of the building set $\B$, the subset $\{\alpha \in \cC_{s_\ell} \colon \alpha \geq \alpha_{2\ell}\}$ of $\cC_{s_\ell}$ is also an element of $\B$,
  and hence, $I_\ell$ is an element of $\widehat{\cP}_\B$.
  Moreover, combining the fact that $\max \cC_{s_\ell} \in \{\alpha \in \cC_{s_\ell} \colon \alpha \geq \alpha_{2\ell}\}$ with \eqref{I_l}, we have that $\max \cC_{s_\ell}$ is contained in $I_\ell$.
  In conclusion,
  $$
   \mu_\B(I_{\ell-1},I_{\ell}) = (\max \fC(I_{\ell};I_{\ell-1}),(\alpha_{2\ell-1},\alpha_{2\ell})) = (\max \cC_{s_\ell},(\alpha_{2\ell-1},\alpha_{2\ell})),
  $$
  where $I_0 = I$.
  Since $s_\ell \leq s_{\ell+1}$, we have that $\max \cC_{s_{\ell}} \geq \max \cC_{s_{\ell+1}}$, and, hence, $\mu_\B(I_{\ell-1},I_{\ell}) \succ \mu_\B(I_{\ell},I_{\ell+1})$.
  We define the maximal chain~$c \coloneqq (I_0,I_1,\ldots,I_N)$ of the interval $[I,J]$, and thus $\mu_\B(c)$ is decreasing.
  
  Consider another maximal chain $c' = (I'_0,I'_1,\ldots,I'_N)$ of $[I,J]$.
  Let $i$ be the smallest index such that $I_i \neq I'_{i}$, and $j$ the smallest index such that $I_{i} \subset I'_{j}$  with $j>i$.
  Since one of $\alpha_{2i-1}$ and~$\alpha_{2i}$ is contained in $\fC(I'_{j};I'_{j-1})$, we have that $\fC(I_{i};I_{i-1}) \subset \fC(I'_{j};I'_{j-1}) \subset \cC_{s_i}$.
  Therefore, 
  $$
    \max \cC_{s_i} = \max \fC(I_{i};I_{i-1}) \leq\max \fC(I'_{j};I'_{j-1}) \leq \max \cC_{s_i},
  $$
  and, hence, $\max \fC(I_{i};I_{i-1}) = \max \fC(I'_{j};I'_{j-1}) = \max \cC_{s_i}$.
  
  Now, let us consider two cases where $\fC(I'_{i};I'_{i-1}) \subset \cC_{s_i}$ or $\fC(I'_{i};I'_{i-1}) \not\subset \cC_{s_i}$.

  \medskip

  \noindent \underline{\textbf{CASE 1. $\fC(I'_{i};I'_{i-1}) \subset \cC_{s_i}$ :}}
  In this case, 
$$
    \max \fC(I'_{i};I'_{i-1}) \leq \max \cC_{s_i} = \max \fC(I'_{j};I'_{j-1}).
$$
If $\max \fC(I'_{i};I'_{i-1}) < \max \cC_{s_i}$, we have that $\mu_\B(I'_{i-1}, I'_i) \prec \mu_\B(I'_{j-1}, I'_j)$.
  Otherwise, \ie
  $$
  \max \fC(I'_{i};I'_{i-1}) = \max \cC_{s_i},
  $$
  we note that $I'_{i} \setminus I'_{i-1}$ intersects with $\cC_{s_i} \setminus I_{i}$, and either $\alpha_{2i-1}$ or $\alpha_{2i}$ is an element of~$I'_{j} \setminus I'_{j-1}$.
  From \eqref{I_l}, we have $\cC_{s_i} \setminus I_{i} = \{\alpha \in \cC_{s_i} \colon \alpha < \alpha_{2i}\}$, and then
  $$
  \min I'_{i} \setminus I'_{i-1} < \alpha_{2i} \leq \max I'_{j} \setminus I'_{j-1},
  $$
  indicating that
  $$
  ((\max I'_{i} \setminus I'_{i-1}, \min I'_{i} \setminus I'_{i-1}),(\max I'_{j} \setminus I'_{j-1},\min I'_{j} \setminus I'_{j-1})) \notin \cR_1.
  $$
  Moreover, since $I'_{i} \setminus I'_{i-1} \cap I'_{j} \setminus I'_{j-1} = \emptyset$,
  $$
  ((\max I'_{i} \setminus I'_{i-1}, \min I'_{i} \setminus I'_{i-1}),(\max I'_{j} \setminus I'_{j-1},\min I'_{j} \setminus I'_{j-1})) \notin \cR_2.
  $$
  Therefore, either they are not comparable or $\mu_\B(I'_{i-1}, I'_i) \prec \mu_\B(I'_{j-1}, I'_j)$, and, hence,
  $\mu_\B(c')$ is not decreasing.
    
  Furthermore, since $\alpha_{2i-1}$ and $\alpha_{2i}$ are the two largest elements of $\cC_{s_i} \setminus I_{i-1}$ and $I_{i-1}=I'_{i-1}$, it can be observed that
  $$ 
  \begin{cases}
    \alpha_{2i-1} = \max I'_{i} \setminus I'_{i-1} > \alpha_{2i} > \min I'_{i} \setminus I'_{i-1}, & \mbox{if } \alpha_{2i-1} \in I'_{i} \setminus I'_{i-1} \\
    \alpha_{2i-1} > \alpha_{2i} \geq \max I'_{i} \setminus I'_{i-1} > \min I'_{i} \setminus I'_{i-1}, & \mbox{otherwise}.
  \end{cases}
  $$
  We conclude that
  $\mu_\B(I_{i-1}, I_{i}) \succ \mu_\B(I'_{i-1}, I'_{i})$,
  signifying that $\mu_\B(c)$ lexicographically succeeds~$\mu_\B(c')$.
  
  \medskip

  \noindent \underline{\textbf{CASE 2. $\fC(I'_{i};I'_{i-1}) \not\subset \cC_{s_i}$ :}}
  Choose $1 \leq t \leq r$ such that the set $\cC_t$ includes the set $\fC(I'_{i};I'_{i-1})$.
  Since
  $$
  \cC_0 \cup \cdots \cup \cC_{{s_i}-1} \subset I_{i-1} = I'_{i-1},
  $$
  we have ${s_i} < t$, and then 
  $$
    \max \fC(I_{i};I_{i-1}) = \max \fC(I'_{j};I'_{j-1})= \max \cC_{s_i} > \max \cC_{t} \geq \fC(I'_{i};I'_{i-1}).
  $$
  Therefore, both $\mu_\B(I_{i-1},I_{i})$ and $\mu_\B(I'_{j-1},I'_{j})$ succeed $\mu_\B(I'_{i-1},I'_{i})$.  
  Hence, $\mu_\B(c')$ is not decreasing, and $\mu_\B(c)$ lexicographically succeeds $\mu_\B(c')$.
  
  In conclusion, by CASES~1 and~2, $\mu_\B$ is reversed edge lexicographical.
\end{proof}

A simplicial complex is \emph{pure} if all its maximal simplices has the same dimension.
A finite pure $d$-dimensional simplicial complex is said to be \emph{shellable} if there exists an ordering $F_1, F_2, \ldots$ of its maximal simplices, say a \emph{shelling}, such that $(\bigcup_{i=1}^{k-1} F_i) \cap F_k$ is pure and of dimension~$d-1$ for each $k \geq 2$.

Let $\cP_\B = \widehat{\cP}_\B \setminus \{\emptyset, I\}$ be a partially ordered set whose order is given by inclusion.
Given that $F$ is a facet of the order complex $\Delta(\cP_\B)$ if and only if $F \cup \{\emptyset,I\}$ is a facet of $\Delta(\widehat{\cP}_\B)$.

\begin{theorem}\label{thm:shellability_of_K_P_for_chordal}
    
    Let $\B$ be a chordal building set on a finite set of cardinality $2k$.
    Then, $\Delta(\cP_\B)$ is shellable, and hence, its geometric realization is homotopy equivalent to a bouquet of $(k-2)$-dimensional spheres; that is,
    $$
        \Delta(\cP_\B) \simeq \bigvee^{\alt(\B)} S^{k-2},
    $$
    where $\alt(\B)$ denotes the number of spheres.
    
\end{theorem}
\begin{proof}
    From \cite{Bjorner-Wachs1982} and \cite{Bjorner-Wachs1983}, if a bounded graded poset $\cP$ is EL-shellable, the order complex $\Delta(\cP)$ of~$\cP$ is shellable.
    By Lemma~\ref{shellable}, the order complex $\Delta(\widehat{\cP}_\B)$ is shellable, and so is $\Delta(\cP_\B)$.
    In addition, it is well known that every shellable complex is Cohen-Macaulay \cite{Stanley1996book}, and then, its geometric realization is homotopy equivalent to a bouquet of spheres of the same dimension.
    Since the length of maximal chains of $\cP_\B$ is $k-1$,  the dimension of each sphere is $k-2$, as desired.
    
\end{proof}

\section{Alternating $\B$-permutations} \label{sec:alt_B_perm}
Let $S$ be a finite set with even cardinality.
This section is devoted to explicitly computing the value of $\alt(\B)$ defined in Theorem~\ref{thm:shellability_of_K_P_for_chordal} for a chordal building set $\B$ on $S$.
A (one-line notation) permutation $x = (x_1x_2 \cdots x_{k})$ on $S$ is called a \emph{$\B$-permutation} if for each $1 \leq i \leq k$, $x_i$ and $\max{\{x_1,x_2,\ldots,x_i\}}$ lie in the same connected component of the restricted building set $\B \vert_{\{x_1,x_2,\ldots,x_{i}\}}$.
A $\B$-permutation $(x_1x_2\ldots x_k)$ is called an \emph{alternating $\B$-permutation} if $x_1 > x_2 < x_3 > \cdots$.
Inspired by the assumption that the cardinality of $S$ is even, we utilize the one-line notation $(x_1x_2/\cdots/x_{2k-1}x_{2k})$ to represent a permutation on $S$, which is grouped into blocks of length two, separated by a slash.

\begin{example}\label{Bperm_example}
    We remark that the number of alternating $\B$-permutations is not an invariant of nestohedra.
    Let us consider two labeling on path graphs with $4$ vertices in Example~\ref{graph_exam}.
    Although their corresponding nestohedra are congruent as polytopes in $\R^4$, but the number of alternating $\B$-permutations are different.
    
    \begin{enumerate}
      \item Consider the labeling on $P_4$ as
      \begin{center}
        \begin{tikzpicture}[scale=1.2]
  \node (1) at (0,0) {$2$};
  \node (2) at (1,0) {$3$};
  \node (3) at (2,0) {$1$};
  \node(4) at (3,0) {$4$};
  
  \draw (1) -- (2);
  \draw (2) -- (3);
  \draw (3) -- (4);
\end{tikzpicture}.
      \end{center} 
    With this labeling, the alternating $\B(P_4)$-permutations are $(31/42)$, $(32/41)$ and $(41/32)$.
    
      \item Consider a path graph $P_4$ of order $4$ labeled as
      \begin{center}
        \begin{tikzpicture}[scale=1.2]
      \node (1) at (0,0) {$1$};
      \node (2) at (1,0) {$2$};
      \node (3) at (2,0) {$3$};
      \node(4) at (3,0) {$4$};
      
      \draw (1) -- (2);
      \draw (2) -- (3);
      \draw (3) -- (4);
        \end{tikzpicture}.
      \end{center}
      There are two alternating $\B(P_4)$-permutations $(21/43)$ and $(32/41)$.
\end{enumerate}
\end{example}

One can immediately obtain the following proposition.

\begin{proposition}\label{no_alter}
    If a building set $\B$ on $S$ has an odd order connected component, then there is no alternating $\B$-permutation.
\end{proposition}

\begin{proof}
    For a $\B$-permutation $x = (x_1x_2/\cdots/x_{2k-1}x_{2k})$ on $S$, there is $1 \leq i \leq k$ such that $x_{2i-1}$ and~$x_{2i}$ are not contained in the same connected component of~$\B \vert_{\{x_1,\ldots,x_{2i}\}}$.
  In other words, $\max \{x_1,\ldots,x_{2i-1}\}$ and $\max \{x_1,\ldots,x_{2i}\}$ are different, and then $x_{2i-1} < x_{2i}$ which confirms that $x$ is not alternating.
\end{proof}

Let $\B$ be a chordal building set on $S$, and $c = (I_0,I_1,\ldots,I_k)$ a maximal chain of~$\widehat{\cP}_\B$.
For each $1 \leq i < k$, we say that $\mu_\B(c)$ has a \emph{decreasing position at $i$} if
$$
  \mu_\B(I_{i-1},I_{i}) \succ \mu_\B(I_i,I_{i+1}).
$$

\begin{lemma}\label{Bstandard}
    Let $\B$ be a chordal building set on $S$ with no connected component of odd order.
    There is a one-to-one correspondence between the set of alternating $\B$-permutations and the set of maximal chains of~$\widehat{\cP}_\B$ with no decreasing position.
\end{lemma}

\begin{proof}
    Let $\cF$ denote the set of alternating $\B$-permutations, and $\cG$ the set of maximal chains $c$ of~$\widehat{\cP}_\B$ such that $\mu_\B(c)$ has no decreasing position.

    Consider an alternating $\B$-permutation $x= (x_1x_2/\cdots/x_{2k-1}x_{2k})$.
    For each $1 \leq i \leq k$, since $x_{2i-1} > x_{2i}$, it follows that $\max \{x_1,\ldots,x_{2i-1}\} = \max \{x_1,\ldots,x_{2i}\}$, and then $x_{2i-1}$ and $x_{2i}$ are contained in the same connected component of $\B$.
    For each $1 \leq i \leq k$, we have
    $$
    I^x_i \coloneqq \{x_1,\ldots,x_{2i}\} \in \widehat{\cP}_\B.
    $$
    For each $1 \leq i < k$, considering the fact that $\max \{x_1,\ldots,x_{2i+1}\}$ and $x_{2i+1}$ belong to the same connected component of~$\B \vert_{\{x_{1},\ldots,x_{2i+1}\}}$, we have
    $$
    \max \fC(I^x_{i};I^x_{i-1}) \leq \max \{x_1,\ldots,x_{2i+1}\} =\max \fC(I^x_{i+1};I^x_{i}).
    $$
    Since $x_{2i} < x_{2i+1}$ and $x_{2i-1} \neq x_{2i+1}$, the pair~$((x_{2i-1},x_{2i}),(x_{2i+1},x_{2i+2}))$ is neither an element of~$\cR_1$ nor~$\cR_2$, which confirms that
    $$
    \mu_\B(I_{i-1},I_{i}) \not\succ \mu_\B(I_i,I_{i+1}).
    $$
    Hence, for each alternating $\B$-permutation $x$, the maximal chain
    $$
    \Phi(x) \coloneqq (\emptyset = I^x_0,I^x_1,\ldots,I^x_k)
    $$
    has no decreasing position, and then $\Phi$ can be regarded as an injective map from~$\cF$ to~$\cG$.

    Let $c = (I_0,I_1,\ldots,I_k) \in \cG$.
    For each $1 \leq i \leq k$, we define the pair~$(x^c_{2i-1},x^c_{2i})$ of $S$ as
    $$
    (x^c_{2i-1},x^c_{2i}) = (\max I_{i} \setminus I_{i-1},\min I_{i} \setminus I_{i-1}).
    $$ 
    Define a permutation $\Psi(c)\coloneqq (x^c_1x^c_2/\cdots/x^c_{2k-1}x^c_{2k})$ on $S$.
    We will prove the following statement by mathematical induction on $k$:
    \begin{equation}\label{induc}
      \text{$\Psi(c)$ is an alternating $\B$-permutation.}
    \end{equation}    
    To begin with, it is obviously confirmed that $\Psi(c) = (x^c_1 x^c_2)$ is an alternating $\B$-permutation.
    Let us assume that \eqref{induc} holds for $k = N-1$, where $N \geq 2$.
    Let $c = (I_0,\ldots,I_{N})$ be a maximal chain of $\widehat{\cP}_\B$ such that $\mu_\B(c)$ has no decreasing position.
    Consider a restricted building set $\B' \coloneqq \B \vert_{I_1 \cup \cdots \cup I_{N-1}}$ of $\B$.
    Note that $c' = (I_0,\ldots,I_{N-1})$ is a maximal chain of $\widehat{\cP}_{\B'}$ and $c'$ also has no decreasing position.
    By the assumption, the permutation $\Psi(c')$ is an alternating $\B'$-permutation.
    To show that $\Psi(c) = (x^c_1x^c_2/\cdots/x^c_{2N-1}x^c_{2N})$ is an alternating $\B$-permutation, it suffices to demonstrate that the following two conditions hold:
    \begin{enumerate}
      \item[(C1)] $x^c_{2N-2} < x^c_{2N-1}$, and
      \item[(C2)] $x^c_{2N-1}$ and $\max \{x^c_1,\ldots,x^c_{2N-1}\}$ lie in the same connected component of $\B \vert_{\{x^c_1,\ldots,x^c_{2N-1}\}}$.
    \end{enumerate}
    Put $x^c_\ell \coloneqq\max \{x^c_1,\ldots,x^c_{2N-2}\}$, where $1 \leq \ell \leq 2N-2$.
    By the assumption, since $\Psi(c')$ is a $\B$-permutation, $x^c_\ell$ and $x^c_{2N-2}$ lie in the same connected component of $\B'$.
    It follows that $x^c_\ell$ is an element of $\fC(I_{N-1};I_{N-2})$, and then, $x^c_\ell = \max \fC(I_{N-1};I_{N-2})$.
    We divide our proof into two distinct cases; $\max \fC(I_{N-1};I_{N-2}) < \max \fC(I_N;I_{N-1})$, or not.

    Consider the first case.
    We have
    $$
        x^c_\ell = \max \fC(I_{N-1};I_{N-2}) < \max \fC(I_{N};I_{N-1}).
    $$
    This inequality implies that $\max \fC(I_{N};I_{N-1})$ is either $x^c_{2N-1}$ or $x^c_{2N}$.
    Considering $x^c_{2N-1} > x^c_{2N}$, we obtain $\max \fC(I_{N};I_{N-1}) = x^c_{2N-1}$.
    In conclusion, the conditions (C1) and (C2) hold for $c$, which confirms that $\Psi(c)$ is an alternating $\B$-permutation.
    
    Now, we consider the case where $\max \fC(I_{N-1};I_{N-2}) \geq \max \fC(I_N;I_{N-1})$.
    Note that $c$ does not have a decreasing position at $(N-1)$, that is,
    \begin{equation}\label{prec}
      \mu_\B(I_{N-2},I_{N-1}) \not\succ \mu_\B(I_{N-1},I_{N}).
    \end{equation}
    It can be observed that $\fC(I_{N-1};I_{N-2}) = \fC(I_N;I_{N-1})$ and $x^c_\ell > x^c_{2N-1} > x^c_{2N}$.
    By the chordality of the building set $\B$, we have that
    $$
    \cC \coloneqq \{x \in \fC(I_N;I_{N-1}) \colon x > x^c_{2N}\}
    $$
    is also an element of $\B$.
    Thus,
    $$
    \{x^c_\ell , x^c_{2N-1}\} \subset \cC \in \B \vert_{\{x^c_1,\ldots,x^c_{2N-1}\}}.
    $$
    Since $x^c_\ell$ is also $\max \{x^c_1,\ldots,x^c_{2N-1}\}$, the condition (C2) holds for $c$.
    By \eqref{prec},
    $$
    ((x^c_{2N-3},x^c_{2N-2}),(x^c_{2N-1},x^c_{2N})) \notin \cR_1,
    $$
    and hence, the condition~(C1) holds for $c$.
    
    By mathematical induction, $\Psi(c)$ is an alternating $\B$-permutation for all $c \in \cG$.
    Since $\Phi(\Psi(c))=c$ for $c \in \cG$, $\Phi$ is surjective, as desired.
\end{proof}

\begin{theorem} \label{thm:a(G)}
    Let $\B$ be a chordal building set.
    Then, $\alt(\B)$ is the number of alternating $\B$-permutations.
\end{theorem}
\begin{proof}
    Let $\B$ be a chordal building set on $S$ of size $2k$.
    There is an one-to-one correspondence between $(\ell-1)$-simplices in $\Delta(\cP_\B)$ and chains of length $\ell$ in $\cP_\B$ for $1 \leq \ell \leq k-1$.
    Then each $(\ell-1)$-simplex can be regarded as an $\ell$-length chain of~$\cP_\B$.
    
    Let an $\ell$-length chain $c = (S_1, S_2, \ldots, S_\ell)$ of $\cP_\B$ be given, where $1 \leq \ell \leq k$.
    For convenience, let $S_0 \coloneqq \emptyset$ and $S_{\ell+1} \coloneqq S$.
    For each $1 \leq i \leq \ell+1$, according to Lemma~\ref{shellable}, there is the unique maximal chain
    $$c^i = (S_{i-1} = I^c_{k_0+\cdots+k_{i-1}}, I^c_{k_0+\cdots+k_{i-1}+1}, \ldots , I^c_{k_0+\cdots+k_{i-1}+k_{i}} = S_i),$$
    of the interval $[S_{i-1},S_{i}]$ of $\widehat{\cP}_\B$ such that $\mu_\B(c^i)$ is decreasing, where $k_0 =0$, and $2k_i = \left\vert S_{i} \setminus S_{i-1} \right\vert$.
    In conclusion, for each $\ell$-length chain $c$ of $\cP_\B$, there is the induced maximal chain 
    $$
    \bar{c} \coloneqq (I^c_0,I^c_1,\ldots,I^c_k)
    $$
    of $\widehat{\cP}_\B$.
    
    Let us consider a maximal chain $d = (I_0,I_1,\ldots,I_{k})$ of $\widehat{\cP}_\B$ with $N$ decreasing positions, where $1 \leq N \leq k-1$.
    Define the set
    $$
    \cI_{d} = \{1 \leq i \leq k-1 \colon \text{ $d$ does not have a decreasing position at $i$.}\}
    $$
    of indices.
    Let $c = (I_{i_1},\ldots,I_{i_\ell})$ be an $\ell$-length chain of $\cP_\B$.
    The induced maximal chain $\bar{c}$ of $\widehat{\cP}_\B$ is equal to $d$ if and only if
    $$
    \cI_{d} \subset \{i_1,i_2,\ldots,i_{\ell}\}.
    $$
    Consequently, the number of $\ell$-length chains of $\cP_\B$ that induces $d$ is given by
    $$
    {N \choose \ell-\left\vert\cI_{d}\right\vert} = {N \choose \ell-(k-1-N)} = {N \choose k-1-\ell},
    $$
    if $k-1-N \leq \ell \leq k-1$, and $0$ otherwise.
    
    Let $\alt_N$ denote the number of maximal chains of $\widehat{\cP}_\B$ with $N$ decreasing positions for each $0 \leq N \leq k-1$.    
    The total number of $\ell$-length chains of $\cP_\B$ is
    $$
    \sum_{N = k-1-\ell}^{k-1}{N \choose k-1-\ell}\alt_N.
    $$
   
    Now, the Euler characteristic $\chi(\Delta(\cP_\B))$ of $\Delta(\cP_\B)$ is
    \begin{align*}
      \chi(\Delta(\cP_\B)) & =\sum_{\ell=1}^{k-1}{(-1)^{\ell-1} \# \ell \text{-length chains of }} \cP_\B \\
      & = \sum_{\ell=1}^{k-1} \sum_{N = k-1-\ell}^{k-1} (-1)^{\ell-1} {N \choose k-1-\ell}\alt_N \\
      & = 
      \alt_{k-1} +\sum_{\ell=0}^{k-1} \sum_{N = k-1-\ell}^{k-1} (-1)^{\ell-1} {N \choose k-1-\ell}\alt_N \\
      & = \alt_{k-1} +  \sum_{N = 0}^{k-1} \left( \sum_{\ell=k-1-N}^{k-1} (-1)^{\ell-1} {N \choose k-1-\ell} \right) \alt_N.
    \end{align*}
    By the binomial theorem, we have that
    $$
    \sum_{\ell=k-1-N}^{k-1} (-1)^{\ell-1} {N \choose k-1-\ell} = 0
    $$
    for all $1 \leq N \leq k-1$, and, hence, $\chi(\Delta(\cP_\B)) = \alt_{k-1} +(-1)^{k-2}\alt_0$.
    
    We note that $\alt_{k-1}=1$ by Lemma~\ref{shellable}. 
    In addition, by Theorem~\ref{thm:shellability_of_K_P_for_chordal}, $\chi(\Delta(\cP_\B)) = 1 + (-1)^{k-2} \alt(\B)$.
    Therefore, $\alt(\B) = \alt_0$, that is the number of alternating $\B$-permutations by Lemma~\ref{Bstandard}.
    
\end{proof}

\section{Proof of the main theorem} \label{sec:proof_of_main_theorem}

In this section, we prove Theorem~\ref{main1}.
For a connected building set $\B$ on $[n+1] =\{1,\ldots,n+1\}$, the rational Betti number of $X^\R_\B$ is determined by the reduced Betti numbers of $(K_\B)_I$ for even cardinality subset $I$ of $[n+1]$ by \eqref{Betti_comp}, where $(K_{\B})_I$ the full subcomplex of $K_{\B}$ induced by all vertices $J$, where $\left\vert J \cap I \right\vert$ is odd.

Let us consider the case when $n+1$ is even, and $I = [n+1]$.
We denote by~$K_{\B}^{\odd}$ (respectively, $K_{\B}^{\even}$) the induced subcomplex of a nested set complex~$K_{\B}$ whose vertices have odd (respectively, even) cardinality.
In other words, $K_{\B}^{\odd}=(K_{\B})_{[n+1]}$ and $K_{\B}^{\even}$ is the complement of the simplicial complex of $K_{\B}^{\odd}$ in $K_\B$.
Since the geometric realization of $K_\B$ is topologically homeomorphic to the $(2k-2)$-dimensional sphere $S^{2k-2}$, by the Alexander duality, $H_\ast(K_{\B}^{\odd})$ is completely determined by $H^\ast(K_{\B}^{\even})$. 
In addition, with the coefficients of (co)homology being in a field $\Q$, by the Universal Coefficient Theorem, we have
\begin{equation}\label{Alex_dual}
  \tilde{H}_i(K_{\B}^{\odd}) \cong \tilde{H}_{2k-3-i}(K_{\B}^{\even}) \text{ for all $i \geq 0$}.
\end{equation}
Refer \cite[Theorem 71.1 and 53.5]{Munkres1984book} for the Alexander duality and the Universal Coefficient Theorem, respectively.

\begin{lemma}\label{geo_real}
    The geometric realizations of $\Delta(\cP_\B)$ and $K_{\B}^{\even}$ are homeomorphic.
    In addition, if $\B$ is chordal, the $i$th reduced rational Betti numbers~$\tilde{\beta}_i$ of $K_{\B}^{\odd}$ are
    $$
    \tilde{\beta}_{i}(K_{\B}^{\odd}) =
    \begin{cases}
      \alt(\B), & \mbox{if } i = k-1,\\
      0, & \mbox{otherwise}.
    \end{cases}
    $$
\end{lemma}

\begin{proof}
  A proof of the former statement of this lemma, limited to graphical building sets, is provided in \cite[Lemma 4.7]{Choi-Park2015}.
  However, upon closer inspection, the proof does not actually rely on the condition of being a graphical building set, confirming that the lemma holds for any building sets.
  The latter statement follows from the preceding statement, Theorem~\ref{thm:shellability_of_K_P_for_chordal} and~\eqref{Alex_dual}.
\end{proof}

Now, let us consider more general cases.
For any simplicial complex $K$, the \emph{link} $\Lk_K(v)$ of a vertex $v$ in $K$ is the set of all faces $\sigma \in K$ such that $v \notin \sigma$ and $\sigma \cup \{v\} \in K$.
It is well known that $\Lk_K(v)$ is a full subcomplex of $K$.
If $\Lk_K(v)$ is a topological cone, then $K$ is homotopy equivalent to the simplicial complex $K - v \coloneqq \{\sigma \setminus \{v\} \colon \sigma \in K\}$.

A simplicial complex $K$ is said to be \emph{flag} if it contains every clique of its $1$-skeleton.
It can be immediately observed that an induced subcomplex of a flag complex is also a flag complex.
For each subset $I$ of $[n+1]$ with even cardinality, we shall discuss the induced subcomplex $(K_{\B})_I$ in order to compute the Betti numbers of $X_{\B}^\R$.
If a nested set complex $K_\B$ is flag, then all induced subcomplexes $(K_{\B})_I$ are also flag complexes.
For each vertex $X$ in $(K_{\B})_I$, a vertex $Y$ in $(K_{\B})_I$ is also a vertex in $\Lk_{(K_{\B})_I}(X)$ if and only if the pair $(X,Y)$ holds the following conditions: 
\begin{equation}\label{3condi}
X \subsetneq Y, Y \subsetneq X,  \text{ or } X \cup Y \notin \B.
\end{equation}

\begin{lemma}\label{remove}
    Assume that a nested set complex $K_\B$ is flag.
    For any $I \subset [n+1]$ with even cardinality, an induced subcomplex $(K_{\B})_I$ is homotopy equivalent to $K_{\B \vert_{I}}^{\text{odd}}$.
\end{lemma}

\begin{proof}
  Define the preorder $\trianglelefteq$ of the set of vertices of $(K_{\B})_I$ as $X_1 \trianglelefteq X_2$ if $X_1 \cap I \subseteq X_2 \cap I$ for each pair $(X_1,X_2)$ of vertices in $(K_{\B})_I$.
  Let $\V$ be the set of vertices in $(K_{\B})_I$ that are not in $I$, with the understanding that $\V$ is nonempty, since the case where $\V = \emptyset$ is straightforward.
  For each $1 \leq i \leq N$, we define $\V_i$ as the set of minimal elements of
  $$
  \mathcal{U}_i = \V \setminus (\bigcup_{0 \leq j \leq i-1} \V_j)
  $$ with respect to the order $\trianglelefteq$, where $\V_0 = \emptyset$ and $N$ is a sufficiently large positive integer.
  
  Now, for each $1 \leq i \leq N$, let $(K_{\B})_I^i$ denote the simplicial complex obtained by removing vertices from $(K_{\B})_I$ belonging to $\bigcup_{0 \leq j \leq i-1} \V_j$.
  Let $X \in \V_i$ be a given.
  Considering that $\left\vert X \cap I\right\vert$ is odd, one can choose a connected component $X_0$ of $\B \vert_{X \cap I}$ with odd cardinality.
  For each vertex $Y$ of~$\Lk_{(K_{\B})_I}(X)$, the pair $(X,Y)$ satisfies \eqref{3condi}.
  To show that either $Y = X_0$ or $Y \in \Lk_{(K_{\B})_I}(X_0)$, we consider \eqref{3condi} by dividing it into three cases.
  \begin{enumerate}
      \item If $X \subsetneq Y$, then $X_0 \subseteq X \subsetneq Y$ which confirms that $Y \in \Lk_{(K_{\B})_I}(X_0)$.
      \item Let $Y \subsetneq X$.
      If $Y \subseteq I$, since $Y \in \B \vert_{X \cap I}$ and $X_0$ is a connected component of $\B \vert_{X \cap I}$, either $Y \subseteq X_0$ or $Y \cup X_0 \notin \B$.
      If $Y \not\subset I$, then $Y \in \V_j$ for $j \geq i$.
      Since $Y \cap I \subseteq X \cap I$ and $X \subset \V_i$, we have $Y \cap I = X \cap I$, and then
      $$
      X_0 \subseteq X \cap I = Y \cap I \subseteq Y.
      $$
      Thus, $Y = X_0$ or $Y \in \Lk_{(K_{\B})_I}(X_0)$.
      \item Let $X \cup Y \notin \B$.
        If $X_0 \cup Y \in \B$, since $X_0 \subset X$, then
        $$
        X \cup Y = X \cup (X_0 \cup Y) \in \B,
        $$ 
        which contradicts that $X \cup Y \notin \B$.
        Hence, $X_0 \cup Y$ is not an element of $\B$.
        It follows that $Y \in \Lk_{(K_{\B})_I}(X_0)$. 
    \end{enumerate}
  In conclusion, the link $\Lk_{(K_{\B})_I}(X)$ of $X$ in $(K_{\B})_I^i$ is a topological cone whose apex is $X_0$.
  Hence, one can remove each elements of $\V_i$ from $(K_{\B})_I^i$, ensuring that its homotopy type remains unchanged.
  
  Therefore, from $\V_1$ to $\V_N$, one can sequentially remove all elements in $\V$ from $(K_{\B})_I$ without changing its homotopy type.
  Since $(K_{\B})_I$ with all elements of $\V$ removed is isomorphic to $K_{\B \vert_{I}}^{\text{odd}}$, we finish the proof of the lemma.
\end{proof}

\begin{lemma}\label{cone}
  Let $K_\B$ be a flag nested set complex.
  For each subset $I \subset [n+1]$ with even cardinality, if $\B \vert_I$ has a connected component of odd order, then $K_{\B \vert_{I}}^{\text{odd}}$ is contractible.
\end{lemma}

\begin{proof}
  Choose a connected component $I_0$ of $\B \vert_{I}$ with odd cardinality.
  Let $Y$ be a vertex of $K_{\B \vert_{I}}^{\text{odd}}$.
  If $Y \subseteq I_0$, then either $Y = I_0$ or $(I_0,Y)$ satisfies \eqref{3condi}.
  
  Thus, we assume that $Y \not\subset I_0$.
  Since $I_0$ is a component of $\B \vert_{I}$, we have $Y \cup I_0 \notin \B \vert_{I}$.
  Hence, $(I_0,Y)$ holds \eqref{3condi}, which confirms that $K_{\B \vert_{I}}^{\text{odd}}$ is a topological cone with its apex is $I_0$, as desired.
\end{proof}

Now, we introduce the main result, restating Theorem~\ref{main1}.

\begin{theorem}
    For a connected chordal building set $\B$ on $[n+1] = \{1,2,\ldots,n+1\}$, the $k$th Betti number of $X^\R_{\B}$ is
    $$
    \beta_{k}(X^\R_{\B}) = \sum_{I \in {[n+1] \choose 2k}}\alt(\B \vert_I),
    $$
    where $\alt(\B \vert_I)$ is the number of alternating $\B \vert_I$-permutations.
\end{theorem}

\begin{proof}

    From \cite[Proposition~9.7]{Postnikov2008}, if a building set $\B$ is chordal, then a nested set complex $K_\B$ is a flag complex.
    Combining \eqref{Betti_comp} with Lemma~\ref{remove}, we have 
    $$
    \beta_k(X_\B^\R) = \sum_{\substack{I \subset [n+1] \\ \left\vert I \right\vert \text{ is even}}} \widetilde{\beta}_{k-1}(K^{\odd}_{\B \vert_I}).
    $$
    
    Since every restricted building set of a chordal building set is chordal, for a subset~$I$ of $\widehat{\cP}_\B$ with $\left\vert I \right\vert = 2k$, we have $\tilde{\beta}_{i}(K_{\B \vert_{I}}^{\text{odd}}) = \alt(\B \vert_I) \cdot \delta_{i,k-1}$ for all integers $i$, where $\delta_{i,j}$ is the Kronecker delta by Lemma~\ref{geo_real}.
  We have
  $$
  \sum_{I \in \hat{\cP}_\B} \widetilde{\beta}_{k-1} (K_\B)_I = \sum_{\substack{I \in \hat{\cP}_\B \\ \left\vert I \right\vert = 2k}} \alt(\B \vert_I).
  $$ 
  Let $I$ be a subset of $[n+1]$ such that $I \notin \widehat{\cP}_\B$.
  By Proposition~\ref{no_alter} and Lemma~\ref{cone}, $\alt(\B \vert_I)$ and~$\tilde{\beta}_i(K^{\odd}_{\B \vert_I})$ both become equal to $0$ for all $i$, respectively.
  Consequently, one can write that
  $$
  \sum_{\substack{I \notin \hat{\cP}_\B \\ \left\vert I \right\vert \text{ is even}}} (K_\B)_I = 0 = \sum_{\substack{I \notin \hat{\cP}_\B \\ \left\vert I \right\vert = 2k}} \alt(\B \vert_I).
  $$

  Therefore,
  $$
  \beta_{k}(X^\R_{\B}) = \sum_{\substack{I \subset [n+1] \\ \left\vert I \right\vert \text{ is even}}} \widetilde{\beta}_{k-1}(K^{\odd}_{\B \vert_I}) =\sum_{\left\vert I \right\vert = 2k} \alt(\B \vert_I)
  $$
for each integers $k$, where $\alt(\B \vert_I)$ is the number of alternating $\B \vert_I$-permutations by Theorem~\ref{thm:a(G)}.
\end{proof}

\begin{remark}
    Let $G$ be a finite simple graph, and let $I$ be a subset of the vertex set~$V(G)$ of~$G$ with even cardinality $2k$. 
    Note that the $k$th Betti number~$\beta_{k}((K_{\B(G)})_I)$ of $(K_{\B(G)})_I$ is the $a$-number $a(G \vert_I)$ of~$G \vert_I$.
    If $G$ is a chordal graph with vertices labeled by a perfect elimination ordering, the number $\alt(\B(G)\vert_I)$ of alternating $\B(G)\vert_I$-permutation is exactly $\beta_{k}((K_{\B(G)})_I)$, which confirms $\alt(\B(G)\vert_I)$ is equal to the $a$-number $a(G\vert_I)$.
\end{remark}

\section{Examples}\label{section4}
Let us consider several examples which illustrate Theorem~\ref{main1}.

\subsection{Permutohedra}
For a complete graph~$K_{n+1}$ with vertices $1,2,\ldots,n+1$, the graphical building set~$\B = \B(K_{n+1})$ consists of all nonempty subsets of $[n+1] =\{1,\ldots,n+1\}$.
Every permutation on $[n+1]$ is also a $\B$-permutation, which proves (1) of Corollary~\ref{graph_corr}.

\subsection{Associahedra}
Let $P_{2k}$ denote a path graph having $2k$ vertices labeled consecutively as $1,2,\ldots,2k$.
One can verify that a permutation $x = (x_1x_2\cdots x_{2k})$ is a $\B(P_{2k})$-permutation if and only if $x$ is \emph{$312$-avoiding} \ie there are no $1 \leq i < j < l \leq 2k$ such that $x_j < x_l < x_i$.
See \cite[Section 10.2]{Postnikov2008} for more details.

\begin{proof}[Proof for (2) of Corollary~\ref{graph_corr}]
    Consider a path graph $G = P_{n+1}$ with $n+1$ vertices labeled consecutively as $1,2,\ldots,n+1$, and a subset $I \subset [n+1]$ whose cardinality is $2k$.
    If $G \vert_I$ has a connected component of odd order, then there is no alternating $\B(G)\vert_I$-permutation by Proposition~\ref{no_alter}.
    
    Assume that $G \vert_I$ has no component of odd order.
    Then $G \vert_I$ consists of path graphs $G_{I_1},\ldots,G_{I_r}$, each of which has an even order $2k_1,\ldots,2k_r$, respectively.
    We assume that $\max I_i + 1 < \min I_{i+1}$ for all $1 \leq i < r$.
    For a $\B(G)\vert_I$-permutation $(x_1x_2\cdots x_{2k})$,
    $$
    I_i = \{x_{2k_0+\cdots+2k_{i-1}+1},x_{2k_0+\cdots+2k_{i-1}+2}, \ldots,x_{2k_0+\cdots+2k_{i-1}+2k_{i}}\},
    $$
    where $k_0 =0$.
    Consequently, we obtain that each $\B(G)\vert_I$-permutation is given by a tuple of $\B(G)\vert_{I_i}$-permutations for all $1 \leq i \leq r$.
    Therefore, the number of alternating $\B(G)\vert_I$-permutations is the product of the numbers of $312$-avoiding alternating permutations on $I_i$ for all $1 \leq i \leq l$.
    
    By applying Theorem~\ref{main1}, we conclude that
    $$
        \beta_k(X^\R_{\B(P_{n+1})}) = \sum_{I \in {[n+1] \choose 2k}} \prod_{1 \leq i \leq r_I} \# 312\text{-avoiding alternating permutations on }I_i,
    $$
    where the induced subgraph $G \vert_I$ consists of the even order connected components $G \vert_{I_1},\ldots,G \vert_{I_{r_I}}$.
\end{proof}

For instance, if we consider a path graph $P_6$ with six vertices labeled consecutively as $1, \ldots, 6$, then the~$k$th Betti number $\beta_k$ of $X^\R_{\B(P_6)}$ is presented in Table~\ref{betti_P6}.
\begin{table}
  \centering
   \begin{tabular}{c|c|c|c}
  &$I$& $312$-avoiding alternating permutations on $I$ & $\beta_k$ \\
  \hline
  $k = 0$ & $\emptyset$ & $\{( \ )\}$ & $1$ \\
  \hline
  \begin{tabular}{@{}c@{}}   \end{tabular} $k =1$& $\{1,2\}$ & $\{(21)\}$ & \begin{tabular}{@{}c@{}}  \end{tabular} \\
  
   & $\{2,3\}$ & $\{(32)\}$ & \\
  
  & \{3,4\} & $\{(43)\}$ & $5$\\
 
  & \{4,5\} & $\{(54)\}$ & \\
  
  & \{5,6\} & $\{(65)\}$ & \\
  \hline
  \begin{tabular}{@{}c@{}}   \end{tabular}  $k =2$& $\{1,2,3,4\}$ & $\{(2143),(3241)\}$ & \begin{tabular}{@{}c@{}}  \end{tabular} \\

   & $\{1,2,4,5\}$ & $\{(2154)\}$ & \\

 & $\{1,2,5,6\}$ & $\{(2165)\}$ & $9$\\

  & $\{2,3,4,5\}$ & $\{(3254),(4352)\}$ & \\

  & $\{2,3,5,6\}$ & $\{(3265)\}$ & \\

  & $\{3,4,5,6\}$ & $\{(4365),(5463)\}$ & \\
  \hline
  $k = 3$ & $\{1,2,3,4,5,6\}$ & $\{(214365),(215463),(324165),(325461),(435261)\}$ & $5$
\end{tabular}
  \caption{Nonzero Betti numbers of $X^\R_{\B(P_6)}$}\label{betti_P6}
\end{table}

\subsection{Stellohedra}
Let $G = K_{1,n}$ denote the star graph with the central vertex~$n+1$ connected to the vertices $1,\ldots,n$.
If $I \in {[n+1] \choose 2k}$ does not include $n+1$ as a vertex, $\B(G)\vert_I$ has a component of odd order, and then there is no alternating $\B(G)\vert_I$-permutation, as stated in Proposition~\ref{no_alter}.

Now, consider the case where $I \in {[n] \choose 2k-1}$.
Let $x= (x_1x_2\cdots x_{2k})$ be an alternating permutation on $I \cup \{n+1\}$.
Then $x$ forms a $\B(G)\vert_I$-permutation if and only if $x_1 = n+1$.
In conclusion, this implies (3) of Corollary~\ref{graph_corr}. 

\subsection{Hochschild polytopes}
In this section, we assume that $(m,n)$ is a pair of non-negative integers such that $m+n$ is even.
Denote by $\Alt_{\Hoch(m,n)}$ the set of alternating permutations~$(x_1x_2 \cdots x_{m+n})$ on $[m+n] =\{1,\ldots,m+n\}$ such that for each $i,j \in [m+n]$;
$$
m+n \geq x_{i} > x_{j} \geq m+1
$$
implies that $i <j$.

Consider a building set $\B_{m,n}$ defined in Section~\ref{Hoch_subsec}.

\begin{proposition}\label{Bmn_perm}
    For a permutation $x$ on $[m+n]$, $x$ is an alternating $\B_{m,n}$-permutation if and only if $x \in \Alt_{\Hoch(m,n)}$.
\end{proposition}
\begin{proof}
Let $\B = \B_{m,n}$.
If $x$ is an element of $\Alt_{\Hoch(m,n)}$, it can be easily verified that $x$ is an alternating $\B$-permutation.
Then we only focus on the part of `only if'.

Let $x = (x_1x_2\cdots x_{m+n})$ be an alternating $\B$-permutation.
Assume that $x$ is not an element of~$\Alt_{\Hoch(m,n)}$.
Then one can choose a pair $(i,j)$ of elements of $[m+n]$ such that
$$
i < j \text{ and } m+1 \leq x_i < x_j \leq m+n.
$$
Note that the singleton $\{x_i\}$ is a connected component of the restricted building set~$\B \vert_{\{x_1,\ldots,x_i\}}$.
Since $x$ is a $\B$-permutation, we have that $x_i = \max \{x_1,\ldots,x_i\}$.
If there exists $1 \leq i'< i$ such that $x_{i'} > x_{i}$, it follows that
$$
\max \{x_1,\ldots,x_{i}\} \geq x_{i'} > x_i.
$$
which contradicts that $x_i = \max \{x_1,\ldots,x_i\}$.
Thus, we assume that $\max \{x_1,\ldots,x_{i}\} = x_i$.
Then either $i = 1$ or $x_{i-1} <x_{i}$. 
Considering that $x$ is alternating, we have that $x_i > x_{i+1}$.
It follows that the singleton $\{x_i\}$ forms also a connected component of~$\B \vert_{\{x_1,\ldots,x_{i+1}\}}$.
We conclude that $x_i$ and $x_{i+1}$ are not contained in the same connected component of $\B \vert_{\{x_1,\ldots,x_{i+1}\}}$, which contradicts $x$ is a $\B$-permutation.
\end{proof}

For non-negative integers $m$ and $n$, by Proposition~\ref{Bmn_perm}, the set $\Alt_{\Hoch(m,n)}$ can be regarded as the set of alternating $\B_{m,n}$-permutations.

\begin{lemma}\label{Hoch_no}
  Let $m$ and $n$ be non-negative integers such that $m+n$ is even.
  If $n > m+2$, there is no alternating $\B_{m,n}$-permutation.
\end{lemma}
\begin{proof}
  Let us consider an alternating $\B_{m,n}$-permutation $x = (x_1x_2\cdots x_{m+n})$.
Note that, for each $1 \leq i \leq \frac{m+n}{2}$, $x_{2i}$ is the maximal element of $\{x_{2i},\ldots,x_{m+n}\}$.
Since $m+n$ is even, we have $n \geq m+4$.
By the pigeonhole principle, there exist $1 \leq \ell < \frac{m+n}{2}$ such that
$$
\{x_{2\ell-1},x_{2\ell}\} \subset [m+1,m+n].
$$
Since $x_{2\ell} > x_{2\ell+1}$, it contradicts the fact that $x$ is alternating.
In conclusion, if $n \geq m+2$, there is no alternating $\B_{m,n}$-permutation.
\end{proof}

Now, we prove Corollary~\ref{Hoch_corr}.

\begin{proof}[Proof for Corollary~\ref{Hoch_corr}]
    Let an integer $0 \leq k \leq \frac{m+n}{2}$ be given.
    For each $1 \leq r \leq \min \{2k,n\}$, let $\cI_r$ be the set defined as 
    $$
    \cI_r = \left\{I \in {[m+n] \choose 2k} \colon I \cap [m+1,m+n]= [m+n-(r-1),m+n]  \right\},
    $$
    and $\cI_0 \coloneqq {[m] \choose 2k}$.
    We note that $\left\vert \cI_r \right\vert = {m \choose 2k-r}$, and for each $I \in \cI_r$, the number of alternating $\B_{m,n} \vert_I$-permutations is exactly $\left\vert \Alt_{\Hoch(s,r)} \right\vert$, where $s = 2k-r$.
    
    For any given $I \in {[m+n] \choose 2k}$, by \eqref{Hoch_build}, the restricted building set $\B_{m,n} \vert_I$ has no connected component of odd order if and only if $I \cap [m+1,m+n]$ is either the empty set or $[m+r,m+n]$ for some $1 \leq r \leq n$.
    Combining Theorem~\ref{main1} with Proposition~\ref{no_alter}, we conclude that
    $$
    \beta_{k}(X^\R_{\Hoch(m,n)}) = \sum_{s+r =2k}{m \choose s}\left\vert \Alt_{\Hoch(s,r)} \right\vert,
    $$
    where $s$ and $r$ are non-negative integers.
    
    Fix a positive integer $m$, and assume that $n \geq m+2$.
    From Lemma~\ref{Hoch_no}, the cardinality of the set
    $$
    \{(s,r) \colon s+r=2k,0 \leq s \leq m, 0 \leq r \leq \min \{n,s+2\}\}
    $$
    is constant.
    For each $n \geq m+2$, we conclude that
    $$
    \beta_{k}(X^\R_{\Hoch(m,n)}) = \beta_{k}(X^\R_{\Hoch(m,m+2)})
    $$
    for all $k \geq 0$.
    \end{proof}
    
    \begin{example}
        Consider the building set $\B_{2,4}$.
        The $k$th Betti number $\beta_k$ of~$X^\R_{\Hoch(2,4)}$ is computed as in Table~\ref{betti_hoch}.
        In addition, the list of the Betti numbers $\beta_k(X^\R_{\Hoch(m,n)})$ of $X^\R_{\Hoch(m,n)}$ up to $m=8$ is given in Table~\ref{tab:appendix}.
        The Python codes used validation are available at 
        \begin{center}
            \url{https://github.com/YounghanYoon/RealHochVar}.
        \end{center}
    \end{example}

\begin{table}
        \centering
        \begin{tabular}{c|c|c|c}
  &$(s,r)$&$\Alt_{\Hoch(s,r)}$ & $\beta_k$ \\
  \hline
  $k = 0$ & $(0,0)$ & $\{( \ )\}$ & $1$ \\
  \hline
  \begin{tabular}{@{}c@{}}   \end{tabular} $k =1$& $(2,0)$ & $\{(21)\}$ & \begin{tabular}{@{}c@{}}  \end{tabular} \\
  
   & $(1,1)$ & $\{(21)\}$ & $\left\vert \Alt_{\Hoch(2,0)}\right\vert+2\left\vert \Alt_{\Hoch(1,1)}\right\vert+\left\vert \Alt_{\Hoch(0,2)}\right\vert = 4$\\
  
  & $(0,2)$ & $\{(21)\}$ & \\
  \hline
  \begin{tabular}{@{}c@{}}   \end{tabular}$k =2$ & $(2,2)$ & $\{(4231),(4132),(2143)\}$ & \begin{tabular}{@{}c@{}}  \end{tabular} \\
  
   & $(1,3)$ & $\{(4132)\}$ & $\left\vert \Alt_{\Hoch(2,2)}\right\vert+2\left\vert \Alt_{\Hoch(1,3)}\right\vert+\left\vert \Alt_{\Hoch(0,4)}\right\vert = 5$\\
  
  & $(0,4)$ & $\emptyset$ & \\
  \hline
  $k = 3$ & $(2,4)$ & $\{(625143),(615243)\}$ & $2$
\end{tabular}
\caption{Nonzero Betti numbers of $X^\R_{\Hoch(2,4)}$}\label{betti_hoch}
\end{table}

\begin{table}
  \begin{tabular}{|c|c|cccccccccc|}
\hline
$m$ & $n$ & $\beta_0$ & $\beta_1$ & $\beta_2$ & $\beta_3$ & $\beta_4$ & $\beta_5$ & $\beta_6$ & $\beta_7$ & $\beta_8$ & $\beta_9$ \\
\hline
$0$ & $\geq2$ & $1$ & $1$ &&&&&&&&\\ 
\hline
$1$ & $2$ & $1$ & $2$ &&&&&&&&\\ 
 & $\geq3$ & $1$ & $2$ & $1$ &&&&&&&\\ 
\hline
$2$ & $2$ & $1$ & $4$ & $3$ &&&&&&&\\ 
 & $3$ & $1$ & $4$ & $5$ &&&&&&&\\ 
 & $\geq4$ & $1$ & $4$ & $5$ & $2$ &&&&&&\\ 
\hline
$3$ & $2$ & $1$ & $7$ & $14$ &&&&&&&\\ 
 & $3$ & $1$ & $7$ & $17$ & $11$ &&&&&&\\ 
 & $4$ & $1$ & $7$ & $17$ & $17$ &&&&&&\\ 
 & $\geq5$ & $1$ & $7$ & $17$ & $17$ & $6$ &&&&&\\ 
\hline
$4$ & $2$ & $1$ & $11$ & $43$ & $33$ &&&&&&\\ 
 & $3$ & $1$ & $11$ & $47$ & $77$ &&&&&&\\ 
 & $4$ & $1$ & $11$ & $47$ & $89$ & $52$ &&&&&\\ 
 & $5$ & $1$ & $11$ & $47$ & $89$ & $76$ &&&&&\\ 
 & $\geq6$ & $1$ & $11$ & $47$ & $89$ & $76$ & $24$ &&&&\\ 
\hline
$5$ & $2$ & $1$ & $16$ & $105$ & $226$ &&&&&&\\ 
 & $3$ & $1$ & $16$ & $110$ & $336$ & $241$ &&&&&\\ 
 & $4$ & $1$ & $16$ & $110$ & $356$ & $501$ &&&&&\\ 
 & $5$ & $1$ & $16$ & $110$ & $356$ & $561$ & $300$ &&&&\\ 
 & $6$ & $1$ & $16$ & $110$ & $356$ & $561$ & $420$ &&&&\\ 
 & $\geq7$ & $1$ & $16$ & $110$ & $356$ & $561$ & $420$ & $120$ &&&\\ 
\hline
$6$ & $2$ & $1$ & $22$ & $220$ & $922$ & $723$ &&&&&\\ 
 & $3$ & $1$ & $22$ & $226$ & $1142$ & $2169$ &&&&&\\ 
 & $4$ & $1$ & $22$ & $226$ & $1172$ & $2949$ & $1982$ &&&&\\ 
 & $5$ & $1$ & $22$ & $226$ & $1172$ & $3069$ & $3782$ &&&&\\ 
 & $6$ & $1$ & $22$ & $226$ & $1172$ & $3069$ & $4142$ & $2040$ &&&\\ 
 & $7$ & $1$ & $22$ & $226$ & $1172$ & $3069$ & $4142$ & $2760$ &&&\\ 
 & $\geq8$ & $1$ & $22$ & $226$ & $1172$ & $3069$ & $4142$ & $2760$ & $720$ &&\\ 
\hline
$7$ & $2$ & $1$ & $29$ & $413$ & $2863$ & $6446$ &&&&&\\ 
 & $3$ & $1$ & $29$ & $420$ & $3248$ & $11507$ & $8651$ &&&&\\ 
 & $4$ & $1$ & $29$ & $420$ & $3290$ & $13327$ & $22525$ &&&&\\ 
 & $5$ & $1$ & $29$ & $420$ & $3290$ & $13537$ & $28825$ & $18186$ &&&\\ 
 & $6$ & $1$ & $29$ & $420$ & $3290$ & $13537$ & $29665$ & $32466$ &&&\\ 
 & $7$ & $1$ & $29$ & $420$ & $3290$ & $13537$ & $29665$ & $34986$ & $15960$ &&\\ 
 & $8$ & $1$ & $29$ & $420$ & $3290$ & $13537$ & $29665$ & $34986$ & $21000$ &&\\ 
 & $\geq9$ & $1$ & $29$ & $420$ & $3290$ & $13537$ & $29665$ & $34986$ & $21000$ & $5040$ &\\ 
\hline
$8$ & $2$ & $1$ & $37$ & $714$ & $7434$ & $32709$ & $25953$ &&&&\\ 
 & $3$ & $1$ & $37$ & $722$ & $8050$ & $46205$ & $95161$ &&&&\\ 
 & $4$ & $1$ & $37$ & $722$ & $8106$ & $49845$ & $150657$ & $108232$ &&&\\ 
 & $5$ & $1$ & $37$ & $722$ & $8106$ & $50181$ & $167457$ & $253720$ &&&\\ 
 & $6$ & $1$ & $37$ & $722$ & $8106$ & $50181$ & $169137$ & $310840$ & $184464$ &&\\ 
 & $7$ & $1$ & $37$ & $722$ & $8106$ & $50181$ & $169137$ & $317560$ & $312144$ &&\\ 
 & $8$ & $1$ & $37$ & $722$ & $8106$ & $50181$ & $169137$ & $317560$ & $332304$ & $141120$ &\\ 
 & $9$ & $1$ & $37$ & $722$ & $8106$ & $50181$ & $169137$ & $317560$ & $332304$ & $181440$ &\\ 
 & $ \geq 10$ & $1$ & $37$ & $722$ & $8106$ & $50181$ & $169137$ & $317560$ & $332304$ & $181440$ & $40320$ \\ 
\hline
\end{tabular}
\caption{Nonzero Betti numbers $\beta_k$ of $X^\R_{\Hoch(m,n)}$ up to $m =8$}
\label{tab:appendix}
\end{table}

\subsection{Cyclohedra} \label{subsection:Cyclohedra}
For $n \geq 5$, let us consider a cycle graph $C_n$ with the vertices $1,2,\ldots,n$. 
    Let any vertex labeling of $C_n$ be given.
    There are two neighbors $x_1$ and $x_2$ of the vertex $1$.
    Without loss of generality, assume that $x_1 > x_2$.
    Choose a neighbor $x_3 \neq 1$ of the vertex $x_2$.
    
    When $I = \{1,x_1,x_2,x_3\}$, the induced subgraph $C_n \vert _{I}$ of $C_n$ is given by
    \begin{center}
        \begin{tikzpicture}[scale=1.2]
  \node (1) at (0,0) {$x_1$};
  \node (2) at (1,0) {$1$};
  \node (3) at (2,0) {$x_2$};
  \node(4) at (3,0) {$x_3$};
  
  \draw (1) -- (2);
  \draw (2) -- (3);
  \draw (3) -- (4);
\end{tikzpicture}.
      \end{center}
Since $1 < x_2 < x_1$, by similarly way of Example~\ref{Bperm_example}, we find that there are exactly three alternating $\B(C_n)\vert_{I}$-permutations.
In contrast, the $a$-number $a(C_n\vert_{I})$ of $C_n\vert_{I}$ is given by
$$
(\#\text{edges of } C_n\vert_{I}) -1
$$
from \cite{Choi-Park2015}, and then $\beta_{1}((K_{\B(C_n)})_I) = 2$ which is different from the number of alternating $\B(C_n)\vert_{I}$-permutations.
Note that $C_n$ is not chordal which implies that $\B(C_n)$ is not a chordal building set for any vertex labeling of $C_n$.

\section{Open problems on the nested set complex}
In this section, we aim to propose several interesting problems based on our study of the nested set complex.

\subsection{Venustus building sets}
Let us briefly summarize all the steps for the main proof.
For a connected chordal building set~$\B$ on $[n+1] = \{1,\ldots,n+1\}$, we have that
\begin{align*}
\beta_k(X^\R_\B) & =\sum_{\substack{I \subset [n+1] \\ \left\vert I \right\vert \text{ is even}}} \widetilde{\beta}_{k-1} ((K_\B)_I)
    \quad (\text{by~\eqref{Betti_comp} as a variant of Theorem~\ref{Choi-Park2017_torsion}})\\
   & =\sum_{\substack{I \subset [n+1] \\ \left\vert I \right\vert \text{ is even}}} \widetilde{\beta}_{k-1} (K_{\B\vert_I}^{\odd})  
    \quad (\text{required the flagness of $K_\B$, and by Lemma~\ref{remove}})\\
   & =\sum_{\substack{I \subset [n+1] \\ \left\vert I \right\vert \text{ is even}}} \widetilde{\beta}_{\left\vert I \right\vert-k-2} (K_{\B\vert_I}^{\even}) 
    \quad (\text{by~\eqref{Alex_dual} as the Alexander duality})\\
   & =\sum_{\substack{I \subset [n+1] \\ \left\vert I \right\vert \text{ is even}}} \widetilde{\beta}_{\left\vert I \right\vert-k-2}(\Delta(\cP_{\B\vert_I}))
    \quad (\text{by Lemma~\ref{geo_real}})\\
   & =\sum_{I \in {[n+1] \choose 2k}} a(\B\vert_I)) 
    \quad (\text{required the chordality of $\B$, and by Theorem~\ref{thm:shellability_of_K_P_for_chordal}})\label{7.5}\\
   & = \sum_{I \in {[n+1] \choose 2k}}{\# \text{alternating $\B\vert_I$-permutations}} 
    \quad (\text{by Theorem~\ref{thm:a(G)}}).
\end{align*}
The one of the important steps in the above proof is Theorem~\ref{thm:shellability_of_K_P_for_chordal}, which concerns the shellability of~$\Delta(\cP_{\B\vert_I})$.
Consequently, the reduced homology of each $(K_\B)_I$ is concentrated in a single degree, with the degree depending only on the size of $I$.
More precisely, if $|I|=2k$, then only $\widetilde{\beta}_{k-1}((K_\B)_I)$ can be positive, and hence, each~$I$ contributes $\beta_k(X_\B^\R)$ of degree~$k$, which is half the cardinality of~$I$.
This phenomenon is advantageous as it simplifies the calculation of Betti numbers for real toric manifolds associated with any chordal building set and allows for a highly structured homology decomposition.
Since this phenomenon does not occur for all building sets, we would like to designate a specific term to those building sets that exhibit this phenomenon.

\begin{definition}
    A connected building set $\B$ on $[n+1]$ is said to be \emph{venustus} if, for each $I \in {[n+1] \choose 2k}$, the reduced homology of $(K_\B)_I$ is concentrated in the degree~$k-1$; that is,
    $$
    \widetilde{\beta}_{i-1}((K_\B)_I) = 0
    $$
    for all $i \neq k$.
\end{definition}

As mentioned in Section~\ref{graph_subsec}, the class of venustus building sets contains not only chordal building sets but also graphical ones, as illustrated in Figure~\ref{fig:venustus building sets}.
\begin{figure}
\centering
\begin{tikzpicture}
    \draw[rounded corners=2mm] (-6,-1.2) rectangle (6.2,3) node [below left] {Building sets};
    \draw[rounded corners=2mm] (1.2,0) rectangle  (-4,1.6) node [below right] {Chordal building sets};
    \draw[rounded corners=2mm] (-1.5,-0.3) rectangle (5.4,1) node [below left] {Graphical building sets};
    \draw[rounded corners=2mm] (5.8,-0.7)rectangle (-5.5,2.3) node [below right] {Venustus building sets};
\end{tikzpicture}
\caption{Venustus building sets}\label{fig:venustus building sets}
\end{figure}

\begin{problem}
  Characterize venustus building sets.
\end{problem}

It would be natural to extend this concept to the class of all real toric manifolds.
Let $X^\R$ be an $n$-dimensional real toric manifold.
Recall that $X^\R$ is completely determined by a pair~$(K,\lambda^\R)$, where $K$ is an $(n-1)$-dimensional simplicial complex and $\lambda^\R$ is a mod~$2$ characteristic map over~$K$.
Let $\omega_{i}$ be the element of $\row \lambda_{\B}^\R$ corresponding to the $i$th row of $\lambda_{\B}^\R$ as a matrix for $i=1, \ldots, n$.
Each nonzero element $\omega = \omega_{i_1} + \cdots + \omega_{i_k}$ of~$\row \lambda_{\B}^\R$ can be regarded as a subset~$\{i_1,\ldots,i_k\}$ of~$2^{[n]}$.
If there is a map $\sigma \colon 2^{[n]} \to [0,n]$ such that
$$
\widetilde{\beta}_{i-1}(K_\omega) = 0 \quad \text{for all $i \neq \sigma(\omega)$},
$$
then we say that the (rational) homology group $H^\ast(X^\R)$ of $X^\R$ is \emph{concentrated in~$\sigma$}.
In particular, a building set $\B$ on $[n+1]$ is venustus if and only if the homology of $X^\R_\B$ is concentrated in $\sigma_n$ defined by
$$
  \sigma_n(\omega) = \begin{cases}
                      \frac{\left\vert \omega \right\vert}{2}, & \mbox{if $\left\vert \omega \right\vert$ is even}, \\
                      \frac{\left\vert \omega \right\vert+1}{2}, & \mbox{if $\left\vert \omega \right\vert$ is odd},
                    \end{cases}
$$
for each $\omega \in 2^{[n]}$.
Motivated by this, one can define for a pair $(K,\lambda^\R)$ to be \emph{venustus} if the corresponding real toric manifold has the homology that is concentrated in~ $\sigma_n$. 

The characterization of venustus pairs seems quite challenging.
Most pairs $(K, \lambda^\R)$ are not venustus, and finding a venustus one is likely to be difficult.
For a pair to be venustus, it undoubtedly needs to have good combinatorial and topological structures for both $K$ and $\lambda^\R$. 

One nontrivial example of venustus pairs is the \emph{real Coxeter variety} $X_{B_n}^\R$ of type~$B_n$.
According to \cite{Choi-Park-Park2017typeB}, the corresponding complexes in the case of $X_{B_n}^\R$ are flag and shellable, which makes the variety venustus.
In contrast, for other types, none are venustus as shown in \cite{Choi-Kaji-Park2019} and \cite{Choi-Yoon-Yu2023}.

Just as a real permutohedral variety $X_{A_n}^\R$ is generalized to a real toric manifold derived from a building set, $X_{B_n}^\R$ can also be extended to ones associated with building sets, as studied in \cite{Devadoss2011}, \cite{PPP2020}, and \cite{Suyama2021}.
It would also be interesting to ask the characterization of \emph{type~$B$ venustus} building sets.
More generally, we propose the following problem.

\begin{problem}
    Characterize venustus pairs $(K, \lambda^\R)$; that is, characterize real toric manifolds whose homology is concentrated in $\sigma_n$.
\end{problem}

\subsection{Unimodality and multiplicative structures}

As shown in Section~\ref{section4}, for certain chordal building sets calculated by the authors so far, it has been observed that the sequence of the Betti numbers of $X^\R_\B$ is unimodal.
In particular, we can rigorously prove that the unimodality holds for real permutohedral varieties. 
Denote by $a_{2k}$ the $2k$th Euler zigzag number (A000111 in \cite{oeis}).

\begin{lemma}[Proposition~3.2 in \cite{Sokal2020}] \label{prop:a/nfac}
  The sequence~$\left\{ \frac{a_{2k}}{(2k)!} \right\}_{k \geq 0}$ is strictly log-concave.
\end{lemma}

\begin{theorem}
    For a fixed positive integer $n$, the sequence $\left\{ \binom{n+1}{2k} a_{2k} \right\}_{k \geq 0}$ is log-concave.
\end{theorem}
\begin{proof}
    Set $\alpha_k = \frac{a_{2k}}{(2k)!}$ and $\beta_k = \binom{n+1}{2k} a_{2k}$.
    Then we have that
    \begin{align*}
        \frac{\beta_{k}^2}{\beta_{k-1}\beta_{k+1}} &= \frac{(n+1-2k+2)(n+1-2k+1) \alpha_k^2}{ (n+1-2k)(n+1-2k-1) \alpha_{k-1}\alpha_{k+1}} \\
        &> \frac{\alpha_k^2}{\alpha_{k-1}\alpha_{k+1}} \\
        &> 1 \quad \text{(by Lemma~\ref{prop:a/nfac}).}
    \end{align*}
\end{proof}
Even if we consider graphical building sets, we have failed to find any example whose Betti number sequence is not unimodal. 
Based on these findings, we propose the following conjecture.
\begin{conjecture}
  For a venustus building set $\B$, the sequence $\{\beta_k(X^\R_{\B})\}_{k \geq 0}$ representing the rational Betti numbers of the associated real toric manifold is unimodal.
\end{conjecture}

On the other hand, for a chordal building set $\B$, one remarkable advantage of our result is that the generator of the cohomology ring of $X^\R_\B$ can also be expressed by alternating $\B$-permutations.
This allows for a nice combinatorial description of multiplicative structures in terms of permutations, similar to a real permutohedral variety \cite{Choi-Yoon2023}.
However, there is a key difference. 
In the case of real permutohedral varieties, $K_{\B\vert_I}^{\odd}$ is a pure simplicial complex, whereas $K_{\B\vert_I}^{\odd}$ is not pure in general.
This difference makes it difficult to directly apply the methods from~ \cite{Choi-Yoon2023} to the case of a chordal building set~$\B$.
Therefore, we leave the following problem.

\begin{problem}
    For a connected chordal building set~$\B$, give an explicit description of a multiplicative structure of rational cohomology ring~$H^\ast(X^\R_\B)$ in terms of alternating $\B$-permutations.
\end{problem}

\end{document}